\documentclass[preprint]{imsart}
\RequirePackage[OT1]{fontenc}
\RequirePackage{amsthm,amsmath}
\usepackage[colorlinks]{hyperref}
\usepackage{hypernat}

\numberwithin{equation}{section}
\usepackage{verbatim}
\usepackage{amssymb}
\usepackage{comment}
\usepackage{appendix}
\usepackage{graphicx}
\usepackage{subfigure}
\usepackage{enumerate}
\usepackage{comment}
\usepackage{microtype}
\usepackage{subfigure}
\usepackage{color} % Need the color package
\usepackage[table]{xcolor}
\usepackage{todonotes}
\usepackage{bbm, dsfont}
\usepackage[left=3cm, right=3cm]{geometry}

\startlocaldefs

\newtheorem{theorem}{Theorem}[section]
\newtheorem{lemma}[theorem]{Lemma}
\newtheorem{proposition}[theorem]{Proposition}
\newtheorem{corollary}[theorem]{Corollary}

\newtheorem{definition}[theorem]{Definition}

\newtheorem{remark}[theorem]{Remark}

\newtheorem{op}[theorem]{Open Problem}

\newcommand{\Ev}{\mathbb{E}}
\newcommand{\E}{\mathbb{E}}

\newcommand{\Pv}{\mathbb{P}}

\newcommand{\CD}{\mathcal {D}}

\newcommand{\CK}{\mathcal {K}}

\newcommand{\CMnD}{\mathrm{CM}_n(\boldsymbol{D})}

\newcommand*{\Zb}{\mathbb Z}
\newcommand*{\Rb}{\mathbb R}

\newcommand*{\ve}{\varepsilon}

\newcommand*{\al}{\alpha}

\newcommand*{\be}{\begin{equation}}
\newcommand*{\ee}{\end{equation}}
\newcommand*{\ba}{\begin{aligned}}
\newcommand*{\ea}{\end{aligned}}
\newcommand*{\barr}{\begin{array}{c}}
\newcommand*{\earr}{\end{array}}

\newcommand{\BIN}{{\sf Bin}}

\newcommand{\CMnd}{\mathrm{CM}_n(\boldsymbol{d})}

\newcommand{\eqn}[1]{\begin{equation}#1\end{equation}}

\newcommand{\sss}{\scriptscriptstyle}
\newcommand{\e}{\mathrm{e}}

\newcommand{\prob}{{\mathbb{P}}}
\newcommand {\convd}{\stackrel{d}{\longrightarrow}}
\newcommand {\convp}{\stackrel{\sss \prob}{\longrightarrow}}

\begin{document}
\begin{frontmatter}
\title{Tight fluctuations of weight-distances \protect\\ in random graphs with infinite-variance degrees}
\runtitle{Tight fluctuations of weight-distances in random graphs}

\begin{aug}

\runauthor{E. Baroni, R. van der Hofstad, J. Komj\'athy}
\author{Enrico Baroni}
\address{Eindhoven University of Technology,\protect\\
\tt{e.baroni@tue.nl}}

\author{Remco van der Hofstad}
\address{Eindhoven University of Technology,\protect\\
\tt{r.w.v.d.hofstad@TUE.nl}}

\author{J\'{u}lia Komj\'athy}
\address{Eindhoven University of Technology,\protect\\
\tt{j.komjathy@tue.nl}}

\affiliation{Eindhoven University of Technology}

\end{aug}
\date{\today}

\begin{abstract}
We prove results for first-passage percolation on the configuration model with i.i.d.\ degrees having finite mean, infinite variance and i.i.d.\ weights with strictly positive support of the form $Y=a+X$, where $a$ is a positive constant. We prove that the weight of the optimal path has tight fluctuations around the asymptotical mean of the graph-distance {\em if and only if} the following condition holds: the random variable $X$ is such that the continuous-time branching process describing first-passage percolation exploration in the same graph with excess edge weight $X$ has a positive probability to reach infinitely many individuals in a finite time. This shows that almost shortest paths in the graph-distance proliferate, in the sense that there are even ones having tight total excess edge weight for various edge-weight distributions.
\end{abstract}



\end{frontmatter}

\maketitle

\section{Introduction}
\subsection{Motivations}
First-Passage Percolation (FPP) has been introduced as a model for the spread of a material in a random medium (see \cite{Hammersley1965}). In more recent times, motivated by the boost in interest in complex networks and the related random graph models for them, it has been involved as a mathematical tool for studying dynamics in complex networks.
A typical setting in this sense is a transportation network in which a flow is carried through (see \cite{2010arXiv1001.2172K}).
The behaviour of the flow depends on the number of edges in the shortest path between the vertices of the network and the passage time cost to cross the edges.
The corresponding mathematical model is a simple and connected graph $G$, such that to every edge $e$ it is assigned a random variable $Y_{e}$ that represents the passage time through the edge $e$, where the edge weights $(Y_e)$ are a collection of independent and identically distributed (i.i.d.) random variables. The main object of study of first-passage percolation is the time that a given flow starting from a vertex $u$ takes to reach a vertex $v$.

In this paper, we study first-passage percolation in the setting of the configuration model random graph, with i.i.d.\ degrees with finite mean and infinite variance.
An important question in the study of FPP regards the geometry of the geodesics, or the time-minimizing paths. For this, we consider three functionals on the weighted graph: the graph-distance between two vertices, i.e., the minimal number of edges between them, the weight-distance, i.e., the minimal total weight between the two vertices along all paths connected them, and the hopcount, that is the number of edges of the minimal-weight path. We focus on the fluctuations of these functionals around their asymptotic mean.
We investigate the general case of i.i.d.\ edge weights of the form $Y=a+X$, where $a$ is a positive constant that we can take to be equal to one without loss of generality. 
In a previous paper \cite{2015arXiv150601255B} the authors have shown in a similar setting, that the weight-distance \cite[Theorem 5] {Janson:arXiv0804.1656}) grows proportionally to $\log\log n$. We now extend this result proving that fluctuations around the mean are tight for both the weight-distance and the hopcount if and only if the random variable $X$ is such that the continuous-time branching process approximation of the local exploration in the given graph with weight $X$ is \emph{explosive}. With this we mean that the process has a positive probability of having infinitely many individuals alive in a finite time. This is a non-universality result in contrast with the case of general edge-weights with a continuous distribution and support that contains $0$ (see \cite{BhaHofHoo12}) and finite variance degrees, in which the hopcount satisfies a central limit theorem. Further, it gives a rather precise picture about the proliferation of almost shortest paths in the graph distance, and thus about the geometry of the configuration model in the infinite-variance degree setting.

\subsection{Notation and organization}
In this section we introduce notation used throughout the paper.
Given two random variables $X$ and $Y$ that are defined on a sample space that is a subset of $\mathbb{R}$, we say that $X\stackrel{d}{=}Y$ if 
$\Pv(X\leq x)=\Pv(Y\leq x)$ for all $x\in \mathbb{R}$.
 With $o_{q}(1)$ we denote a sequence of real numbers such that $o_{q}(1)\rightarrow 0$ as $q\rightarrow \infty$.
 A sequence of random variables $(X_n)_{n\geq 1}$ converges in probability to a random variable $X$, and we write $X_{n}\stackrel{\Pv}{\rightarrow}X$ if, for all $\varepsilon>0$, $\Pv(|X_{n}-X|>\varepsilon)\rightarrow 0$. 
A sequence of random variables $(X_n)_{n\geq 1}$ converges to $X$ in distribution, and we write $X_{n}\stackrel{d}{\rightarrow}X,$ if 
 $\\\lim_{n\rightarrow\infty}\Pv(X_{n}\leq x)=\Pv(X\leq x)$ for all $x\in \mathbb{R}$ for which $F_{\sss{X}}(x)=\Pv(X\leq x)$ is continuous.
  We denote the set $\{1,2,\dots, n\}$ by $[n]$.
  A sequence of events $\mathcal{E}_{n},n\in\mathbb{N}$ is said to hold with high probability (w.h.p.) if
 \mbox{$\lim_{n\rightarrow\infty}\Pv(\mathcal{E}_{n})=1$.} 
  $\BIN(n,p)$ denotes the random variable with binomial distribution where the number of trials is $n,$ and the probability of success is $p$.
$\CMnd$ denotes the configuration model on $n$ vertices with degree sequence $\boldsymbol{d}$.
 Let $(X_{n})_{n\geq 1}$ be a sequence of random variables and let $\Pv$ be a probability measure, then we say that 
 $(X_{n})_{n\geq 1}$ is tight if for all $\varepsilon>0$, there exists $r>0$ such that $\sup_{n\geq 1}\Pv(|X_{n}|>r)<\varepsilon$.
 
 This section is organized as follows. In Section \ref{themodel}, we introduce the model. In Section \ref{sec-res}, we describe our results.
 We close this section in Section \ref{sec-overview} by giving an overview of the proof and relating it to the literature.
 
\subsection{The model}
\label{themodel}
Our setting is the configuration model $\CMnd$ (see \cite{Bollobas01}) on $n$ vertices, where $\boldsymbol{d}=(d_1,\dots, d_n)$, and $d_{i}$ is the degree vertex $i\in[n]$.
Further, let $\mathcal{L}_{n}=\sum_{i\in[n]}d_{i}$ denote the total degree. The configuration model random graph is obtained as follows: we assign to vertex $i$ a number $d_{i}$ of half-edges and we pair these half-edges uniformly at random. When two half-edges are paired they form an edge, and we continue the procedure until there are no more half-edges available.
If $\mathcal{L}_{n}$ is odd, we add a half-edge to vertex $n$; this extra half-edge makes no difference to our results and we will refrain from discussing it further. 
We consider i.i.d.\ degrees with cumulative distribution function $x\mapsto F(x)$ satisfying 
	\be
	\label{slowlypower}
	 x^{-\tau+1-C(\log x)^{\gamma-1}}\leq 1-F(x)\leq  x^{-\tau+1+C(\log x)^{\gamma-1}},
	\ee	
for $\gamma\in(0,1)$, $C>0$ and $\tau\in(2,3)$. We assume that $F(1)=0$, so that $\min_{i\in[n]} d_{i}\geq 2$ a.s.
The condition  on the minimal degree guarantees that almost all the vertices of the graph lie in the same connected component (see \cite[Proposition 2.1]{KHB13}), or, equivalently, the giant component has size $n(1-o(1))$ (see \cite[Theorem 2.2]{2016arXiv160303254F}).
All edges are equipped with i.i.d.\ edge weights with distribution function $F_{\sss Y}(y)=\Pv(Y\le y)$, where $Y=a+X$ is a non-negative random variable with $\inf\mathrm{supp}(X)=0$, $a>0$, and $X$ has a continuous distribution. Without loss of generality, we can assume $a=1$.
We call $X$ the {\em excess weight} of the edge.
When $d_{i}$ are i.i.d.\  from a distribution $D$ with distribution function $F$ satisfying \eqref{slowlypower}, let $B$ be defined as the (\emph{size-biased version} of $D$)-1, that is,
	\begin{align}
	\label{forwarddegree}
	\Pv(B=k):=\frac{k+1}{\E[D]}\Pv(D=k+1).
	\end{align}
	
We consider several distances and functionals on the graph:
\begin{definition}[Distances in graphs]
\label{edgedistance}
Given two vertices $u$ and $v$, the {\em graph-distance} $\mathcal{D}_{n}(u,v)$  is the minimal number of edges on a path connecting $u$ and $v$.
The \emph{passage-time} or \emph{weight-distance} \mbox{from $u$ to $v$} is defined as
 	\be\label{passagetime}
	W_{n}(u,v):=W_{n}^{\sss Y}(u,v)=\min_{\pi\colon u\rightarrow v}\sum_{e\in \pi}Y_{e},
	\ee
where the minimum is over all paths $\pi$ in $G$ connecting $u$ and $v$, and $Y:=(Y_{e})_{e}$ is the collection of the edge weights.
The \emph{hopcount} $H_{n}(u,v)$ is the number of edges in the smallest-weight path between $u$ and $v$.
\end{definition}
\medskip

The main aim of this paper is to study these functionals on the configuration model with i.i.d.\ degrees having distribution function $F$ satisfying \eqref{slowlypower} and i.i.d.\ edge weights $Y=1+X$, where $\inf\mathrm{supp}(X)=0$.

\subsection{Results}
\label{sec-res}
Our main result is the following theorem:
	\begin{theorem}[Tightness criterion for excess edge-weights]
	\label{integralcondition}
	Consider the configuration model with i.i.d.\ degrees having distribution function $F$ satisfying \eqref{slowlypower} for some $\tau\in(2,3)$, and let $u$ and $v$ be chosen uniformly at random from $[n]$. Suppose that the edge weights are i.i.d.\ and are of the form $Y=1+X$, where $ \inf\mathrm{supp}(X)=0$ and $X$ has cumulative distribution function $F_{\sss X}(x)$ with its generalised inverse $F_{\sss X}^{(-1)}(y):=\inf\{t\in \Rb\colon  F_{\sss X}(t)\ge y\}$ that satisfies that, for some $\ve,C>0,$
	\be
	\label{inverse-crit}
	\int_{1/\ve}^\infty F_{\sss X}^{(-1)}\left(\e^{-Cu}\right)\frac1u\mathrm du <\infty.
	\ee
Then
	\be\label{weightdisres1}
 	W_{n}(u,v)-\frac{2\log\log{n}}{|\log(\tau-2)|}
 	\ee
is a tight sequence of random variables. Consequently, 
	\be
	\label{tightness-hopcount}
	H_{n}(u,v)-\frac{2\log\log{n}}{|\log(\tau-2)|}
	\ee
is tight sequence of random variables.
\end{theorem}
\medskip

The tightness of the hopcount in \eqref{tightness-hopcount} follows easily, since, for our choice of edge weights,
	\be 
	\mathcal{D}_{n}(u,v)\leq H_{n}(u,v)\leq W_{n}(u,v).
	\ee
Therefore, by \eqref{weightdisres1} in Theorem \ref{integralcondition} and the tightness of 
	\be
\mathcal{D}_{n}(u,v)-\frac{2\log\log n}{|\log(\tau-2)|}
	\ee	
as proved in \cite[Theorem 1.2]{MR2318408}, it follows that \eqref{tightness-hopcount} holds.

\begin{remark}\label{explosiveconservative}\normalfont
The integral condition in \eqref{inverse-crit} is equivalent to the following. Let $B$ be defined in \eqref{forwarddegree} and $X$ be the excess edge-weight. Let  $(h_{\sss B},F_{\sss X})$ be the age-dependent branching process where individuals have random life-lengths with distribution $F_{\sss X}(t)$. At death, every individual produces a family of random size with offspring distribution $F_{\sss B}$. Here $h_{\sss B}(s)$ denotes the probability generating function of the distribution  $F_{\sss B}$. 
Then, \eqref{inverse-crit} holds {\em if and only if} this age-dependent branching process is \emph{explosive}, meaning that there is a positive probability that $N_{t}=\infty$, where $N_{t}$ denotes the number of individuals alive at some finite time $t>0$. Otherwise, the process is called conservative. See \cite{amini2013} or \cite[Section 6]{2016arXiv160201657K} for the proof of this result.
\end{remark}
\medskip

We next investigate what happens when the criterion in \eqref{inverse-crit} fails:

\begin{proposition}[Non-tightness of excess edge-weight]
\label{iff}
Consider the configuration model with i.i.d.\ degrees from distribution $F$ satisfying \eqref{slowlypower} with $\tau\in(2,3)$. If condition \eqref{inverse-crit} does not hold, then
	\be
	W_{n}(u,v)-\frac{2\log\log{n}}{|\log(\tau-2)|}\convp\infty.
	\ee
\end{proposition}
\medskip

To prove Proposition \ref{iff} we state a useful lemma:
\begin{lemma}\label{conservativeweight}
Consider the configuration model with i.i.d.\ degrees from distribution $F$ satisfying \eqref{slowlypower} with $\tau\in(2,3)$, and let $u$ and $v$ be chosen uniformly from $[n]$. Suppose we have i.i.d.\ weights given by the random variable $X,$ where $X$ does not satisfies condition \eqref{inverse-crit}. Then
	\be
 	W_{n}^{\sss X}(u,v)\convp\infty.
	\ee 
\end{lemma}

\begin{proof}
The statement is a consequence of the branching process approximation of the neighborhood of a vertex in $\CMnd$ and the hypothesis that the process is conservative, see  \cite{2015arXiv150601255B} and \cite[(11) and Proposition 3]{2016arXiv160201657K}.
\end{proof}
\medskip

Now we can prove Proposition \ref{iff}:

\begin{proof}[Proof of Proposition \ref{iff}]
Let $Y_{e}=X_{e}+Z_{e}$, where $Y_{e},X_{e}$ and $Z_{e}$ are edge weights, then, for any pair of vertices $u_{0}$ and $v_{0}$ in the same connected component,
	\be\label{weightineq}
	W_{n}^{\sss Y}(u_{0},v_{0})\geq W_{n}^{\sss X}(u_{0},v_{0})+W_{n}^{\sss Z}(u_{0},v_{0}),
	\ee
since
	\be
	\sum_{e\in\pi^{\sss Y}}X_{e}+Z_{e}\geq \sum_{e\in\pi^{\sss X}}X_{e}+\sum_{e\in\pi^{\sss Z}}Z_{e},
	\ee
where $\pi^{\sss Y},\pi^{\sss X}$ and $\pi^{\sss Z}$ are the minimal-weight paths for the edge-weights $(X_{e})_{e}$, $(Y_{e})_{e}$ and $(Z_{e})_{e},$ respectively.
For $Z_{e}=1$, $W_{n}^{\sss Z}(u_{0},v_{0})=\mathcal{D}_{n}(u_{0},v_{0})$. If $u,v$ are two uniformly chosen vertices, then we rewrite \eqref{weightineq} as
	\be
	W_{n}^{\sss Y}(u,v)-\frac{2\log\log n}{|\log(\tau-2)|}\geq W_{n}^{\sss X}(u,v)+\Big(\mathcal{D}_{n}(u,v)-\frac{2\log\log n}{|\log(\tau-2)|}\Big),
	\ee
and, as proved in \cite[Theorem 1.2]{MR2318408}, $\mathcal{D}_{n}(u,v)-\frac{2\log\log n}{|\log(\tau-2)|}$ is tight. If $X_{e}$ is such that condition \eqref{inverse-crit} 
in Theorem \ref{tightdist} is not satisfied, then, by Lemma \ref{conservativeweight},
 	\be
 W_{n}^{\sss X}(u,v)\convp\infty,
	 \ee
	  which implies 
 	\be
 W_{n}^{\sss Y}(u,v)-\frac{2\log\log n}{|\log(\tau-2)|}\convp\infty.
 	\ee
Thus $W_{n}(u,v)-\frac{2\log\log n}{|\log(\tau-2)|}$ is not a tight sequence of random variables.
\end{proof}
\medskip

We next merge Proposition \ref{iff} and Theorem \ref{integralcondition} in a single theorem. For this, let $F_{\sss X}$ be the cumulative distribution for the random variable $X$, let $D$ be the degree distribution  and $B$ the random variable defined in \eqref{forwarddegree}. If $(h_{\sss B},F_{\sss X})$ is the modified age-dependent branching process defined in Remark \ref{explosiveconservative}, we can define the following sets:
	\be\label{eq:univ-exp}
	\mathcal{E}(D)=\{F_{\sss X}\text{ s.t. }(h_{\sss B},F_{\sss X})\text{ is explosive}\},
	\ee
and
	\be
\mathcal{T}(D)=\{F_{\sss X}\text{ s.t. }W_{n}^{\sss (1+X)}-\frac{2\log\log n}{|\log(\tau-2)|} \text{ is tight}\}.
	\ee

\begin{theorem}[Universality class for tightness total excess edge-weight]
\label{thm-univ}
If $D$ satisfies the power-law condition in \eqref{slowlypower} with $\tau\in(2,3)$, then
	\be
	\mathcal{E}(D)=\mathcal{T}(D).
	\ee                                                
\end{theorem}
\begin{proof}                                                                                                                                                                                     
The proof follows directly from Proposition \ref{iff}, Theorem \ref{integralcondition} and Remark \ref{explosiveconservative}.
\end{proof}

\begin{remark}[Universality and proliferation of almost shortest paths]
\label{univ}
\normalfont
Theorem \ref{thm-univ} shows that, as in the case of edge weights $(Y_e)_{e}$ with $\inf\mathrm{supp}(Y)=0$, (see \cite[Theorem 4]{2015arXiv150601255B}) there are two universality classes in the case where the edge-weights take the form $Y=1+X$, where $ \inf\mathrm{supp}(X)=0$. Remarkably, these two universality classes are the same. This result shows that there are {\em many} almost shortest paths in the graph distance metric, since we can even find one with tight total excess edge-weight for many distributions of the excess edge weights. However, when the excess edge-weights have too thin tails close to zero, such that \eqref{inverse-crit}, such paths can no longer be found. This gives us a much more complete picture of the geometry of the configuration model with infinite-variance degrees, and brings the discussion of the universality classes of FPP on it substantially further.
\end{remark}
%\medskip

Theorems \ref{integralcondition} and \ref{thm-univ} leave open whether the fluctuations converge in distribution. That is part of the following open problem:

\begin{op}[Weak convergence of fluctuations]
\label{op-weak-conv}
Let $D$ satisfy the power-law condition in \eqref{slowlypower} with $\tau\in(2,3)$. Show that if condition \eqref{inverse-crit}  is satisfied, then
	\be
	W_{n}^{\sss Y}(u,v)-\frac{2\log\log n}{|\log(\tau-2)|}\convd W_{\sss \infty}
	\ee
for some limiting random variable $W_{\sss \infty}$.  If condition \eqref{inverse-crit} is not satisfied, is          
	\be
	W_{n}^{\sss Y}(u,v)-\frac{2\log\log n}{|\log(\tau-2)|}
	\ee     
of the same order of magnitude as $W_{n}^{\sss X}(u,v)$?                     
\end{op}

\subsection{Overview of the proof}
\label{sec-overview}
In this section, we describe the key ingredients in the proof. 

\subsubsection{Upper tightness of the graph-distance $\mathcal{D}_n(u,v)$}
\label{tightdist}
In the first part of the proof we consider a slightly different setting from the one in Section \ref{themodel}. Instead of i.i.d.\ degrees, we consider general fixed degree sequences for which the empirical degree distribution satisfies the lower bound in \eqref{slowlypower}.
In this setting, we prove a uniform upper bound on the difference between the graph-distance  of two vertices of sufficiently high degree and $2\log\log{n}/|\log(\tau-2)|$.
For this we construct a path that has length less than $2\log\log{n}/\log(\tau-2)|$ plus a tight random variable.
The construction is as follows: we start from vertex $u_{k}$ with degree at least $k$, for a fixed but large constant $k$. Then we find a sequence of interconnected sets $\Gamma_{i}:\Gamma_{i}\supset\Gamma_{i+1}$, where $0\leq i\leq b(n),\ \Gamma_{i}=\left\{v\colon d_v\geq y_{i}\right\}$, and $b(n)$ is less than $\log \log n/|\log (\tau-2)|$, for some increasing sequence $y_i$. In more detail, we show that for any fixed small $\varepsilon>0$ there exists an increasing sequence $y_{i}$ such that the following properties hold: $y_0=k$ and a vertex in $\Gamma_{i}$ is connected to at least one vertex in $\Gamma_{i+1}$ with probability at least $1-\varepsilon_{i}$,
and $\sum_{i=0}^{b(n)}\varepsilon_{i}<\varepsilon$. Moreover, the last set of the sequence $\Gamma_{b(n)}$ is a subset of the complete graph formed by the vertices of high degree, by which we mean degree at least $n^{1/2+\delta}$ for some small $\delta>0$. Further, $b(n)$ is a tight random variable away from $\log\log{n}/|\log(\tau-2)|$.  %thus the sequence of sets provides a connection between the vertex $u_{k}$ and the maximal degree vertex of length that is only a tight random variable away from $\log\log{n}/|\log(\tau-2)|$.
So, starting from two uniformly chosen vertices $u_{k}$ and $v_{k}$, the two paths constructed with the procedure above, we connect $u_{k}$ and $v_{k}$ to the same complete graph in a number of steps that is at most $2\log\log{n}/|\log(\tau-2)|$ plus a tight random variable.

\subsubsection{Tightness of the weight-distance $W_n^{\sss Y}(u,v)$ via degree-dependent percolation}
To prove the tightness of the weight-distance, the rough idea is the following: our goal is to modify the construction of the path for the graph-distance so that between each two consecutive layers $\Gamma_i, \Gamma_{i+1}$, only edges with smaller and smaller excess edge-weight are allowed, in such a way that the total excess edge-weight is summable.
To make this idea rigorous, we extend the construction by Janson (see \cite{Janson:arXiv0804.1656}) to a degree-dependent  percolation on the configuration model keeping each half-edge incident to a vertex of degree $d$ with a probability $p(d)$, a function of the degree. (As we remark on in more detail below, we also make sure that the resulting degree sequence, after this thinning, still satisfies \eqref{slowlypower} with $\tau\in(2,3)$.)
We choose this function $p(d)$ in such a way that it can be expressed as $\Pv(X\le \xi_d)$, for some appropriately chosen sequence $\xi_d$ that depends on the distribution of the excess edge-weight $X$. Then, we can view percolation of the half-edges as follows: we keep a half-edge $s$ attached to a vertex with degree $d$ if and only if the excess edge-weight on the half-edge $s$ satisfies $X_s\le \xi_d$.
Note that in Theorem \ref{integralcondition} we have assumed that the weight on an edge is of the form $Y=1+X$, while the degree percolation assigns a weight to each half-edge.  To solve this issue, we solve the case when the edge weights are of the form $Y'=1+X_{1}+X_{2}$ where $X_{1}$ and $X_{2}$ are two i.i.d.\ random variables from distribution $X$. Then, by a stochastic domination argument, $1+ X\ {\buildrel d \over \le}\  1+ X_1+ X_2$, thus, distances in the graph with edge-weights $Y$ are stochastically smaller than distances in the graph with edge weights $Y'$. Further, both distances are bounded from below by the graph-distance. Hence, tightness of the weight-distance with respect to the edge-weights from distribution $Y'$ implies  tightness of the weight-distance with respect to the edge-weights from distribution $Y$.
Starting from an i.i.d.\ degree distribution satisfying \eqref{slowlypower} for some $\tau\in(2,3)$, we prove that under the condition that \eqref{inverse-crit} is satisfied, it is possible to choose $\xi_d$ and thus $p(d)$ in such a way that the new percolated graph has a new degree sequence with an empirical degree distribution that still satisfies a lower bound as the one in \eqref{slowlypower} with the same exponent $\tau$, that is, the setting that allows the construction of the path described in Section \ref{tightdist}.
Namely, to construct a path from $u$ to $v$ with bounded excess edge-weight, we use two steps. First, we approximate a constant-size neighborhood of the vertices $u,v$ by two branching processes to reach two vertices that have degree at least $k$, for some large but fixed constant $k$, in the new percolated graph. Then, in the second step, we connect $u_k$ to $v_k$ within
 the \emph{percolated graph} with a bounded excess edge-weight. Thus,
	\be
	W_{n}(u,v)\leq W_n(u, u_k)+ W_n(v, v_k) +  \mathcal{D}_{n}^p(u_k,v_k)+\sum_{x\in\pi^{*}}2\xi_{d_{x}},
	\ee
where $\mathcal{D}_n^p(\cdot, \cdot)$ denotes the graph-distance within the percolated graph and $\xi_{d_{x}}$ is the upper bound on the excess edge-weight on the half-edges that are attached to $x$. Here $x$ is  a vertex on the constructed path that we denote by $\pi^\star$. The first two contributions are clearly tight, the second is by \cite[Theorem 1.2]{MR2318408}, while the final contribution can be seen to be tight for an appropriate choice of the $(\xi_{d})$ {\em precisely} when \eqref{inverse-crit} is satisfied. This describes the structure of the proof.

\subsection{Discussion and related problems}
First-passage percolation has been studied extensively in different settings, starting from the grid $\Zb^{2}$ to a wide variety of random graphs, including the configuration model.
One of the main problems in first-passage percolation regards the typical weight-distance between two points in the graph. Moreover, if we assume that the edges have a passage-time represented by a collections of i.i.d.\ random variables, a second problem is to determine the geometry of the time-minimizing paths between two points and the way in which they differ from graph-distance paths. A third problem regards the nature of the fluctuations of these distances and of the hopcount around their asymptotic mean values.

In the context of random graphs, these questions have been widely investigated, for instance in \cite{CPC:46717}, Janson proves that on the complete graph $K_{n}$ with i.i.d.\ exponential weights, the weight-distance between two points grows asymptotically as $\log n/n$. Bhamidi, the second author and Hooghiemstra in \cite{BhaHofHoo12} examine the Erd{\H o}s-R\'enyi random graph $G_{n}(p_{n})$ with i.i.d.\ exponential edge weights. When $np_{n}\rightarrow\lambda>1$, the weight-distance centered by a multiple of $\log n$ converges in distribution, while, when $np_{n}\rightarrow\infty$ they prove that the graph-distance is of order $o(\log n)$, and that the addition of edge weights changes the geometry of the graph.
The same  authors show in \cite{BHH10} that on the configuration model, when the degree sequence has  finite variance with an extra logarithmic moment, then first-passage percolation has only one universality class in the sense that $W_n(u,v)-\gamma_n\log{n}$ converges in distribution for some $\gamma_n\rightarrow \gamma>0$, while $H_n(u,v)$ satisfies an asymptotic Central Limit Theorem with asymptotic mean and variance proportional to $\log{n}$. 
In \cite{EskHofHoo06}, van den Esker, the second author and Hooghiemstra generalize the results on configuration model with finite variance degree to a more general class of random graphs, including the Erd{\H o}s-R\'enyi random graph, showing that the fluctuations around the asymptotic mean are tight.

The setting of configuration models with power-law degrees having infinite asymptotic variance has also been investigated:
In \cite{2015arXiv150601255B}, we prove results for the weight-distance for i.i.d.\ edge weights $X$ with $\inf\mathrm{supp}(X)=0$. We have shown the existence of two universality classes, one corresponding to explosive weights as in \eqref{eq:univ-exp} and one corresponding to conservative weights. These two classes correspond to different weight-distances, in particular, in the explosive case, the weight converges to the sum of two i.i.d.\ random variables.
A result on the nature of fluctuation is given by the second author, Hooghiemstra and Znamenski in \cite{MR2318408}, where they proved that the graph-distance for power-law exponent $\tau\in(2,3)$ and i.i.d.\ degrees centers around $2\log\log n/|\log (\tau-2)|$, and that the fluctuations are tight.

\paragraph{Organization of this paper} In Section \ref{sec-tightness-gd} we prove upper tightness of  $\mathcal{D}_{n}(u,v)-2\log\log n/|\log (\tau-2)|$ under relatively weak assumptions on the degrees, which is a crucial ingredient in our proof. In Section \ref{sec-deg-perc} we combine this result with a degree-dependent half-edge percolation argument to find the graph distance in our percolation graph. In Section \ref{sec-tightness} we complete the proof tightness of the weight-distance.

\section{Tightness of the graph-distance}
\label{sec-tightness-gd}
In this section, we consider $\CMnd$ with a deterministic degree sequence $\boldsymbol{d}$ such that the empirical degree distribution satisfies the lower bound in \eqref{slowlypower}. In this setting we prove that the difference between the graph-distance and $2\log\log n/|\log(\tau-2)|$ is uniformly bounded from above. More precisely, we show the following:
\begin{proposition}[Upper tightness of graph distances]
\label{uppertight}
Given $\CMnd,$ where the degree sequence $\boldsymbol{d}=(d_1,\dots, d_n)$ has empirical distribution function $F_{n}(x)$. Suppose that there exists $\alpha>1/2$ such that for all $x\in[x_{0},n^{\alpha})$, 
	\be\label{condtight}
	1-F_{n}(x)\geq  \frac{c}{x^{\tau-1-C(\log x)^{\gamma-1}}},
	\ee 
with $\mathcal{L}_{n}=\sum_{i\in[n]}d_{i}\leq n\beta$ for some positive $\beta$ and a given $x_{0}>0$. Then
for all $\varepsilon_{\ref{uppertight}}>0$, there exists $k=k(\varepsilon_{\ref{uppertight}})\in\mathbb{N}$ s.t., when $u$ and $v$ are two uniformly chosen vertices in $[n]$ with degree at least $k$, conditionally on being in the same connected component,
	\be\label{kbound}
	\sup_{n\geq 1}\Pv\Big(\mathcal{D}_{n}(u,v)-\frac{2\log\log n}{|\log (\tau-2)|}\geq 1~\Big\vert~ d_{u}\geq k, d_{v}\geq k \Big)<\varepsilon_{\ref{uppertight}}.	
	\ee
\end{proposition}

\begin{remark}[I.i.d.\ vs.\ non-i.i.d.\ degrees]
\normalfont
Note that in Proposition \ref{uppertight}, the degrees of the vertices do not necessarily have to be i.i.d. For the i.i.d.\ case satisfying \eqref{slowlypower}, tightness of $\mathcal{D}_{n}(u,v)-2\log\log n/|\log (\tau-2)|$ has already been proved in \cite{MR2318408}. The extension to the non-i.i.d.\ case is crucial in our analysis.
\end{remark}
Before proving Proposition \ref{uppertight} we state a definition:

\begin{definition}[Nested layers]
\label{def-nested-layers}
Given a graph $G=(V,E)$, a sequence of subsets of $V$, $(A_{i})_{i\in I}$, with $I=[n]$ for some $n$ or $I=\mathbb{N}$, is a \emph{nested sequence of layers} in $G$ if the following two properties hold:
\begin{enumerate}
\item[(1)] $A_{i}\supseteq A_{i+1}$ for all $i\in I$;
\item[(2)] For all $v\in A_{i},$ there exists a vertex $w\in A_{i+1}$ such that $(v,w)\in E$.
\end{enumerate}
\end{definition}
By definition, given a nested sequence of layers in the graph, for any pair $(i,j)$ with $i,j\in I$ and $i<j$, we can define a path of length $j-i$ from any 
vertex in $A_{i}$ to some vertex in $A_{j}$.

\subsection{Proof of Proposition \ref{uppertight}}
To prove Proposition \ref{uppertight} we find a nested sequence of layers in $\CMnd$ intersecting the set of vertices of high degree in $\CMnd$. For this, given a sequence $y_{i}$ of positive integers, we define $\Gamma_{y_{i}}=\left\{v\colon d_v\geq y_{i}\right\}.$
Let 
	\be 
	\Lambda_{\alpha}=\left\{v\colon d_v\geq n^{\alpha}\right\}.
	\ee
Our aim is to prove that for any small $\varepsilon>0$, there exists an increasing sequence $\{y_{i}\}_{i\in I}$, with $y_{0}=k$ s.t. $\Gamma_{y_{i}}$ has the following properties:
\begin{enumerate}\label{nestedseq}
\item $\Gamma_{y_{i}}$ is asymptotically a nested sequence of layers with probability at least $1-\varepsilon$;
\item $|I|<\log\log n/|\log(\tau-2)|$;
\item  $\Gamma_{|I|}$ is w.h.p. connected with the set $\Lambda_{\frac{1}{2}+s}$ for $s>0$ small enough.
\end{enumerate}
For this, we choose a sequence $\{y_{i}\}_{i\geq 0}$ with $y_{0}=k$ that satisfies the following: let $u_{i}$ be a vertex chosen according to the size-biased distribution in $\Gamma_{y_{i}}$ and $E_{i}:=\{u_{i} \text{ is not connected to } \Gamma_{y_{i+1}}\}$. Then $\sum_{i=0}^{b(n)}\Pv(E_{i})<\varepsilon/2$, where $b(n)=|I|$ is less than $\log \log n/|\log(\tau-2)|$ with probability at least $1-\varepsilon/2$.
Let us write $\Pv_{n}(A)=\Pv(A\mid D_{1},\dots D_{n})$ for any event $A$.
We want to prove that
	\be\label{varepsiloni}
	\lim_{k\rightarrow\infty}\sum_{i=0}^{\infty}\Pv_{n}(u_{i}\in\Gamma_{y_{i}},u_{i}\not\to\Gamma_{y_{i+1}}\mid \deg(u_{0})\geq k)=0,
	\ee
 where $u_{i}$ is chosen according to a size-biased distribution from $\Gamma_{y_{i}}$.
Let $S_{y_{i}}$ be  the number of half-edges and $V_{y_{i}}$ be the number of vertices in $\Gamma_{y_{i}}$, respectively. Then 
	\be
	V_{y_{i}}=n(1-F_{n}(y_{i})),
	\ee
and
	\be
	S_{y_{i}}\geq y_{i}n(1-F_{n}(y_{i})).
	\ee 
We consider the sequence $y_{i}$ with $y_{0}=k$. Then,
	\be
	\Pv_n(E_{i})\leq \left(1-\frac{S_{y_{i+1}}}{\mathcal{L}_{n}}\right)^{y_{i}/2}\leq \exp\left\{-\frac{y_{i+1}y_{i}[1-F_{n}(y_{i+1})]}{n\beta}\right\}.
	\ee
Here $\mathcal{L}_{n}\leq \beta n$ is the total number of half-edges in the graph and the factor $y_{i}/2$ in the exponent comes from the worst-case scenario in which we connect all the half-edges of $u_{i}$ to $u_{i+1}$.
We want to prove that
	\be\label{firstcond}
	\sum_{i=0}^{\infty}\exp\Big\{-\frac{y_{i+1}y_{i}[1-F(y_{i+1})]}{n\beta}\Big\}=o_k(1).
	\ee	
For this we introduce the shorthand notation for the absolute value of the exponent in \eqref{firstcond},
	\label{gi}
	\be
	g_{i}(y_{i}, y_{i+1}):=\frac{y_{i+1}y_{i}[1-F(y_{i+1})]}{n\beta}.
	\ee
%We need to choose the sequence $y_{i}$ so that $\sum_{i}\e^{-g_{i}(y_{i}, y_{i+1})}<\infty.$ 
Using \eqref{slowlypower}, we bound $g_{i}$ as
	\be\label{gii}
	g_{i}(y_{i}, y_{i+1})\geq \tilde{c}y_{i+1}^{2-\tau  -C(\log(y_{i+1}))^{\gamma-1}}y_{i},
	\ee
where $C$ is defined in \eqref{slowlypower} and $\tilde{c}>0.$
We would like to choose the sequence $(y_i)_{i\ge 1}$ so that \eqref{firstcond} holds and then 
we can choose $k=k(\varepsilon)$ in Proposition \ref{uppertight} large enough so that
\be\label{sum:gi}	
\sum_{i=0}^{b(n)}\Pv_{n}(u_{i}\in\Gamma_{y_{i}}\not\to\Gamma_{y_{i+1}}\mid \deg(u_{0})>k)\leq\sum_{i=0}^{b(n)}\e^{-g_{i}(y_{i}, y_{i+1})}<\frac{\varepsilon}{2}
	\ee
	holds, where $b(n)=|I|$ is the total number of layers involved. We claim that a sequence satisfying these conditions is given recursively by
	\be\label{recursive}
	y_{0}=k,\quad y_{i+1}={y_{i}}^{(\tau-2+B(\log(y_{i}))^{\gamma-1})^{-1}},
	\ee
with $\gamma$ as in \eqref{slowlypower} and $B>0$ to be defined later on.
We give upper and lower bounds on $(y_{i})_{i\geq 0}$ in the following lemma:
\begin{lemma}\label{yibounds}
For every $\delta>0$ small enough, there exists $k\in \mathbb{N}$ such that, if $y_0\geq k$,
	\be\label{logbound}
	k^{({\frac{1}{\tau-2+\delta}})^{i}}\leq y_i \leq k^{({\frac{1}{\tau-2})}^{i}}.
	\ee
\end{lemma}
\begin{proof}

Note that the sequence $y_i$ is monotone increasing, while, if $y_0>k_0$ for $k_0$ sufficiently large, it holds
 	\be\label{monoth}
 \tau-2+\frac{B}{(\log y_{i})^{1-\gamma}}< \tau-2+\frac{B}{(\log y_{0})^{1-\gamma}}<1.
 	\ee
For a choice of $y_{0}$ satisfying the second inequality in \eqref{monoth}, we define $\delta:=\frac{B}{(\log y_0)^{1-\gamma}}$.
We now get a lower bound on $y_{i}$ using
	\be
	y_{i+1}\geq y_{i}^{\frac{1}{\tau-2+\delta}}\geq \dots \geq y_{0}^{(\frac{1}{\tau-2+\delta})^{i}}=k^{(\frac{1}{\tau-2+\delta})^{i}},
	\ee
while an upper bound is obtained by omitting the term $B\log (y_{i})^{\gamma-1}$, which is non-negative, in the exponent, so that, recursively,
	\be
	y_{i+1}\leq y_i^{\frac{1}{\tau-2}}\leq \dots \leq y_0^{(\frac{1}{\tau-2})^{i}}= k^{(\frac{1}{\tau-2})^{i}}.
	\ee
This concludes the proof of Lemma \ref{yibounds}.
\end{proof}
\medskip

We now prove that  \eqref{firstcond} holds for $\{y_i\}_{i\geq 0}$ as in \eqref{recursive}:
\begin{lemma}\label{summabilitygi}
For an appropriate choice of $B$ in \eqref{recursive}, with $y_{0}=k,$
	\be
	\lim_{k\rightarrow\infty}\sum_{i=0}^{\infty}\e^{-g_i(y_{i},y_{i+1})}=0.
	\ee
\end{lemma}

\begin{proof}
We use the lower bound on $g_{i}(y_{i},y_{i+1})$ in \eqref{gii},
	%\be
 	%g_i(y_{i},y_{i+1})\geq \tilde{c}y_{i+1}^{2-\tau  -C(\log(y_{i+1}))^{\gamma-1}}y_{i},
 	%\ee
 and we replace $y_{i+1}$ with the recursion in \eqref{recursive}. Then we obtain
  	\be
	g_{i}(y_{i},y_{i+1})\geq \tilde{c}y_{i}^{-\frac{\tau-2+C(\log (y_{i+1}))^{\gamma-1}}{\tau-2+B(\log (y_{i}))^{\gamma-1}}+1}.
	\ee
Since $\gamma<1$, using the lower bound on $y_{i}$ in \eqref{logbound}, we see that $B(\log y_{i})^{\gamma-1}<\delta'$ for some $\delta'< 1-(\tau-2)$. Then, elementary calculation yields that
		\be
		g_{i}(y_{i},y_{i+1})\geq \tilde{c}y_{i}^{\frac{B(\log y_{i})^{\gamma-1}-C(\log y_{i+1})^{\gamma-1}}{\tau-2+\delta'}}.
		\ee
We investigate the numerator on the rhs. Since $y_{i}$ is monotone increasing, for the choice of $B>2C$ we get that
	\begin{equation*}
	B(\log y_{i})^{\gamma-1}-C(\log y_{i+1})^{\gamma-1}\geq 2C(\log y_{i})^{\gamma-1}-C(\log y_{i})^{\gamma-1}\geq C(\log y_{i})^{\gamma-1}.
	\end{equation*}
Using this bound in the numerator, and then the lower bound on $y_{i}$ in Lemma \ref{logbound}, we obtain
	\be
	g_{i}(y_{i},y_{i+1})\geq \tilde{c}\exp\Big\{C(\log y_{i})^{\gamma}\Big\}\ge  \tilde{c}\exp\Big\{C(\log k)^{\gamma}\Big(\frac{1}{\tau-2+\delta}\Big)^{i\gamma}\Big\}.
\ee
%Then, using the lower bound for $y_{i}$ in Lemma \ref{logbound}, we obtain
%	\begin{align}
%	g_{i}(y_{i},y_{i+1})%&\geq \tilde{c}\exp\Big\{C(\log y_{i})^{\gamma}\Big\}\\
%	%&\geq \tilde{c}\exp\Big\{C\Big(\frac{1}{\tau-2+\delta}\Big)^{i}(\log k)^{\gamma}\Big\}\\
%	&\geq \tilde{c}\exp\Big\{C(\log k)^{\gamma}\Big(\frac{1}{\tau-2+\delta}\Big)^{i\gamma}\Big\}.
%	\end{align}
Note that $\tau-2+\delta<1,$ so that $\exp\{-g_{i}(y_{i},y_{i+1})\}$ is summable in $i$.
Finally, also note that, since $\gamma>0$ and $y_{0}=k$,
		\be 
		\sum_{i=0}^{\infty}\exp\{-g_{i}(y_{i},y_{i+1})\}\rightarrow 0,
		\ee 
as $k\rightarrow \infty$, establishing the statement of Lemma \ref{summabilitygi} as well as \eqref{firstcond}.
\end{proof}
%So, $k$ can be chosen large enough such that
% 	\be\label{eq:lemwithve}
% 	\sum_{i=0}^{b(n)}\Pv_{n}(u_{i}\in\Gamma_{y_{i}}\not\to\Gamma_{y_{i+1}}\mid \deg{u_{0}}>k)\leq\sum_{i=0}^{b(n)}\e^{-g_{i}(y_{i},y_{i+1})}\le \ve/2.	\ee
We want to give an upper bound on $b(n)$, i.e., the index $i$ for which $\Gamma_{i}$ is a subset of $\Lambda_\al$, the set of vertices of high degree. 
For this, we refine the lower bound in Lemma \ref{yibounds} in the next lemma:

\begin{lemma}\label{stim}
There exists $\delta>0$ such that, for $k$ sufficiently large,
	\be
	y_{i}\geq (y_{0}^{1-\delta})^{(\tau-2)^{-i}}.
	\ee
\end{lemma}

\begin{proof}
Let $y_{i+1}=y_{0}^{a_{i+1}}$, so that by \eqref{recursive},
%	\be
%	y_{i+1}=y_{0}^{a_{i+1}}=(y_{i}^{a_{i}})^{\frac{1}{\tau-2+B(\log y_{i})^{\gamma-1}}},
%	\ee
%which implies
	\be
	a_{i+1}=\frac{a_{i}}{\tau-2+B(\log y_{i})^{\gamma-1}}=\dots=\prod_{j=0}^{i}\Big(\tau-2+B(\log y_{j})^{\gamma-1}\Big)^{-1}.
	\ee
Using the lower bound on $y_i$ in Lemma \ref{yibounds} we get
	\be\label{233}
	a_{i+1}\geq (\tau-2)^{-(i+1)}\prod_{j=0}^{i}\Big( 1+(\tau-2+\delta)^{-j(\gamma-1)}\frac{B(\log k)^{\gamma-1}}{\tau-2}   \Big)^{-1}.	
	\ee
 The convergence of the second product on the rhs of \eqref{233} for any fixed $k$
%	\be
%	\prod_{j=1}^{i+1}\Big( 1+ (\tau-2+\delta)^{j(1-\gamma)}\frac{B(\log k)^{1-\gamma}}{\tau-2}\Big)
%	\ee 
is equivalent to the convergence of the series  
	\be 
	(B\log k)^{1-\gamma}\sum_{j=0}^{i}(\tau-2+\delta)^{j(1-\gamma)},
	\ee 
and this series converges because $\tau-2+\delta<1$ for $\delta>0$ sufficiently small.
So, let us write
	\be
	M_{k}^{-1}:=\lim_{i\rightarrow\infty}\prod_{j=1}^{i}\Big( 1+ (\tau-2+\delta)^{j(1-\gamma)}\frac{(B\log k)^{1-\gamma}}{\tau-2}\Big),
	\ee
and since the partial products on the rhs increase to the limit $M_k^{-1}$, we obtain the lower bound
	\be
	a_{i+1}\geq \frac{1}{(\tau-2)^{i+1}}M_k.	
	\ee
Further observe that due to $\gamma<1$, $M_k=1+o_{k}(1)$, therefore, using the form $y_{i}=y_0^{a_i}$ again, 
	\be
	y_{i}\geq \Big(y_{0}^{M_{k}}\Big)^{\frac{1}{(\tau-2)^{i}}}\geq (y_{0}^{1-\delta})^{\frac{1}{(\tau-2)^{i}}}	
	\ee
for some $\delta>0$ that can be taken arbitrarily small by taking $k$ so large that $M_k\geq 1-\delta$.
\end{proof}

Now we have all the preliminaries to complete the proof of Proposition \ref{uppertight}.
\begin{proof}[Proof of Proposition \ref{uppertight}]
By the condition in Proposition \ref{uppertight}, there are some vertices in the graph of degree $n^{\al}$, for $\al>1/2$. 
As a consequence of Lemma \ref{stim}, the number of layers needed to reach the highest degree vertices in $\Lambda_\al$ has $b(n)$ as an upper bound\footnote{This is an upper bound since the true number is an integer, while the solution of \eqref{eq:ibar-def} is not necessarily}, where $b(n)-1$ is the solution of
	\be\label{eq:ibar-def}
	(y_{0}^{1-\delta})^{(\frac{1}{\tau-2})^{i}}=n^{\al},
	\ee
that is, by elementary calculations,
	\be\label{stepmax}
	b(n)\le \frac{\log\Big(\log(n^{\al})/\log y_{0}^{1-\delta}\Big)}{\log(\frac{1}{\tau-2})}+1= \frac{\log\log n+\log \al-\log(\log (y_{0}^{1-\delta}))}{|\log(\tau-2)|}+1.
	\ee
Then, for $y_0=k$ sufficiently large,
	\be
	b(n)\leq \frac{\log\log n}{|\log(\tau-2)|}.	
	\ee
By Lemma \ref{summabilitygi}, for all $k\ge k_0(\ve)$,
	\be\label{connect-111}
	\sum_{i=0}^{b(n)}\Pv_{n}(u_{i}\in\Gamma_{y_{i}}\not\to\Gamma_{y_{i+1}})\le \sum_{i=0}^{\infty}\e^{-g_{i}(y_i,y_{i+1})}=o_k(1).
	\ee
As a consequence of \eqref{connect-111}, w.h.p., we can connect the vertex $u$ and $v$ with degree $k$, to the set $\Lambda_{\alpha}$ with probability $1-\ve/2$ in at most $b(n)$ steps.
As a consequence of Lemma  \cite[Lemma 5]{BhaHofHoo12}, $\Lambda_{\alpha}$ is a complete graph, we can connect $u_{k}$ and $v_{k}$ in $2b(n)+1$ steps.
We then have
	\be
	\Pv_n(\mathcal{D}_{n}(u,v)-\frac{2\log\log n}{|\log (\tau-2)|}>1\mid d_{u}\geq k, d_{v}\geq k)<o_{k}(1),	
	\ee
that is
	\be
	\sup_{n\geq 1}\Pv_n(\mathcal{D}_{n}(u,v)-\frac{2\log\log n}{|\log (\tau-2)|}>1|d_{u}\geq k, d_{v}\geq k)<o_{k}(1).	
	\ee
We choose $k$ in such a way that $o_{k}(1)<\varepsilon_{\ref{uppertight}}$.
This ends the proof of Proposition \ref{uppertight}.
\end{proof}

\section{Degree percolation}
\label{sec-deg-perc}
In this section our goal is to define degree-dependent percolation on the configuration model. This means that we keep each edge with a probability that depends on the degree of the two vertices the edge is adjacent to. In what follows, we explain two different ways to do this, and show that they are in fact equivalent.
Let
	\be\label{pd}
	p(d):\mathbb{N}\longrightarrow[0,1]
	\ee
be a decreasing monotone function of $d$. Later $p(d)$ will equal probability of keeping a half-edge that is attached to a vertex with degree $d$.
 We now define two different degree-dependent percolation methods given a function $p(d)$.
 If $s$ denotes a half-edge, then we shortly write $p(s)=p(d_{v(s)})$, where $v(s)$ is the vertex that the half-edge $s$ is attached to and $d_{v(s)}$ is its degree.
\begin{definition}[Degree-dependent half-edge percolation]\label{halfper}
Given a half-edge $s$ and a degree sequence $\boldsymbol{d}=(d_{1},\dots,d_{n})$, we keep the half-edge $s$ with probability $p(s)$, independently of all other half-edges.
Further, on the event of not keeping a half-edge $s$, we create a new ``artificial'' vertex of degree $1$ with one half-edge that corresponds to $s$.
We call the original vertices ``regular" and the new vertices ``artificial", and we say that a half-edge is ``regular" or  ``artificial" if it is  attached to a 
regular vertex or to an artificial vertex, respectively.

After this procedure is performed for all half-edges in the graph, we start pairing  the regular half-edges as in the configuration model, and their pairs could be both regular or artificial. When all the regular half-edges have been paired, from the graph that we obtained we remove the artificial vertices, the edges connected to them and the artificial half-edges that have not been paired. We call $\CMnd_{p(d)}$ the obtained graph.
Let $\boldsymbol{d^{r}}=({d}_{1}^{r},\dots,{d}_{n}^{r})$ be the sequence of regular half-edges attached to the regular vertices and let $\boldsymbol{d^{rr}}=({d}_{1}^{rr},\dots,{d}_{n}^{rr})$ be the final degrees of the vertices in $\CMnd_{p(d)}$, that is, the degree sequence in which we count only those regular half-edges that are paired to regular half-edges.
\end{definition}

\begin{definition}[Degree-dependent half-edge percolation]\label{edgeper}
We start with $\CMnd$ (that is the resulting graph after all the half-edges have been paired) and we independently keep an edge $(u,v)$ with probability $p(d_{u})p(d_{v})$, where $d_{u}$ and $d_{v}$ are the degrees of $u$ and $v$.
We call $\widetilde{\mathrm{CM}}_n(\boldsymbol{d})_{p(d)}$ the resulting graph.
\end{definition}

In what follows, our goal is to show that the two different percolation methods result in two random graphs that have the same distribution.
As a preparation to show this, we will calculate the probability of a matching to occur in both percolation methods. We adapt an argument by Janson \cite{Janson:arXiv0804.1656}, who studies various types of percolation on random graphs, including degree-dependent site percolation. Let $\boldsymbol{d}=(d_1,\dots, d_n)$ be a fixed degree sequence, and recall that ${\mathcal{L}}_{n}=\sum_{i\in[n]}d_{i}$ denotes the total degree. Let $1\le s_i \le {\mathcal{L}}_{n}$ be different half-edges  and let us define
	\be
	\label{matching}
	M:=\{(s_{1},s_{2}),(s_{3},s_{4}),\dots ,(s_{2k-1},s_{2k})\},
	\ee
with $2k\leq {\mathcal{L}}_{n}$ a matching, i.e., a sequence of pairs of half-edges.
We would like to calculate the probability of seeing the matching $M$ in $\CMnd_{p(d)}$ and also that in $\widetilde{\mathrm{CM}}_n(\boldsymbol{d})_{p(d)}$. Note that in $\CMnd_{p(d)}$ all the half-edges in $M$ must be regular, so
 we call such a matching a \emph{regular matching}.
Let $E_{\sss M}$ be the event to have the matching $M$ in the graph model under consideration.
We denote by $\Pv$ the measure on the $\sigma-$algebra generated by $\CMnd_{p(d)}$ and $\widetilde{\Pv}$ the measure on the $\sigma-$algebra generated by $\widetilde{\mathrm{CM}}_n(\boldsymbol{d})_{p(d)}$. We have the following equality in law:

\begin{lemma}
Fix a degree sequence $\boldsymbol{d}=(d_1,\dots, d_n)$. Then, for any matching $M$ as in \eqref{matching},
	\be
	\widetilde{\Pv}(E_{\sss M})=\Pv(E_{\sss M}\cap \{M\text{is a regular matching}\}).
	\ee
\end{lemma}

\begin{proof}
Note that the matching $M$ as in \eqref{matching} does not specify the order in which the half-edges are paired.
So, let $2k\leq {\mathcal{L}}_{n}$ and  $1\leq s_{i}\leq {\mathcal{L}}_{n}$ be an ordered sequence of half-edges
and let us define
	\begin{equation*}
	E_{k}(M):=(E_{s_{1}s_{2}})_{1}\cap (E_{s_{3}s_{4}})_{2}\dots \cap (E_{s_{2k-1}s_{2k}})_{k}
	\end{equation*}
to be the event that half-edges $s_{i}$ are matched to each other in exactly this order.
We let
	\be
	E_{k}^{\sss R}(M):=E_{k}\cap\{s_{1},s_{2},\dots, s_{2k}\text{ are regular half-edges}\}
	\ee
be the event that $E_k(M)$ happens and all the half-edges in the matching are regular.
We want to show that $\Pv(E_{k}^{\sss R}(M))=\widetilde{\Pv}(E_{k}(M))$
for any $k\leq {\mathcal{L}}_{n}/2$.
First we calculate the probability of $E_k^{\sss R}(M)$ in the degree-dependent half-edge percolation model, that is in $\CMnd_{p(d)}$. We calculate the probability by induction. To initialize, we start with $k=1$ and we compute
	\be
	\Pv(E_{1}^{\sss R}(M))=\Pv\big((E_{s_{1}s_{2}})_{1}\big)=\Pv(\{s_{1}, s_{2} \text{\,regular\,} \}\cap\{s_{1} \text{\,is paired to\,} s_{2}\})
	=p(s_{1})p(s_{2})\frac{1}{{\mathcal{L}}_{n}-1},
	\ee
where we have used that the probability that $s_i$ is regular is $p(s_i)$, independently for all half-edges $s_i$. Further, we have also used that the total number of half-edges is still ${\mathcal{L}}_{n}$ after the degree-dependent half-edge percolation, because of the addition of the artificial vertices. Now notice that, for general $k$,
	\be
	\Pv(E_{k}^{\sss R}(M))=\Pv(E_{k}^{\sss R}(M)\mid E_{k-1}^{\sss R}(M))\Pv(E_{k-1}^{\sss R}(M))=\frac{1}{{\mathcal{L}}_{n}-2k-1}p(s_{2k-1})p(s_{2k})\Pv(E_{k-1}^{\sss R}(M)).
	\ee
Indeed, once we have the matching $E_{k-1}^{\sss R}(M)$, the next two half-edges $s_{2k-1}$ and $s_{2k}$ must be regular: this happens with probability $p(s_{2k-1})p(s_{2k})$, and these two half-edges must be paired to each other. Since we already paired $2k-2$ half-edges, the conditional probability of pairing $s_{2k-1}$ to $s_{2k}$ is $1/({\mathcal{L}}_{n}-2k-1)$.
We obtain that
	\be
	\label{pekr} 
	\Pv(E_{k}^{\sss R}(M))= \prod_{i=1}^{2k} p(s_i) \prod_{i=1}^{k} \frac{1}{{\mathcal{L}}_{n}-2i-1}.
	\ee
We now determine the probability of $E_k(M)$ in $\widetilde{\mathrm{CM}}_n(\boldsymbol{d})_{p(d)}$.
Notice that here we do the matching first, so that the $i$th pair of half-edges are matched to each other with probability $1/({\mathcal{L}}_{n} - 2i -1)$. Then, conditionally that the edge $(s_{2i-1}, s_{2i})$ is formed, we perform the degree-dependent half-edge percolation on this matching independently between different edges. Thus, the probability that we keep the edge $(s_{2i-1}, s_{2i})$ is precisely $p(s_{2i-1})p(s_{2i})$, independently for different values of $i$.
As a result,
	\be	
    	\widetilde{\Pv}(E_{k}(M))=\prod_{i=1}^{k}\frac{1}{{\mathcal{L}}_{n}-2k-1} \prod_{i=1}^{k}p(s_{2i-1})p(s_{2i}),
	\ee
which is exactly the same as the formula in \eqref{pekr}. This finishes the proof.
\end{proof}
We obtain the following immediate corollary:
\begin{corollary}[Equality in law of degree-dependent percolations]
Let $p(d)$ be a function as in \eqref{pd} and let $\mathrm{CM}_n(\boldsymbol{d})_{p(d)}$, $\widetilde{\mathrm{CM}}_n(\boldsymbol{d})_{p(d)}$ be the resulting graphs of the degree-dependent half-edge percolation as described in Definition \ref{halfper} and the degree-dependent edge percolation as in Definition \ref{edgeper}, respectively. Then
$\mathrm{CM}_n(\boldsymbol{d})_{p(d)}\overset{d}{=}\widetilde{\mathrm{CM}}_n(\boldsymbol{d})_{p(d)}$.
\end{corollary}

In the next results, our goal is to determine the new degree distribution of the regular vertices, that is, the empirical degree distribution of $\boldsymbol{d^{rr}}=(d_1^{rr}, d_2^{rr}, \dots, d_n^{rr})$.
Note that we start with a degree sequence $\boldsymbol{d}=(d_1, \dots, d_n)$ with a power-law empirical distribution that satisfies \eqref{slowlypower}. We would like to  maintain a similar power-law condition. This is of course not possible for an arbitrary choice of $p(d)$, thus we need to restrict the edge-retention probabilities to satisfy some degree-dependent bounds. The next proposition is about this:

\begin{proposition}\label{csh}
Let $\CMnd$ be a configuration model with i.i.d.\ degrees following a distribution that satisfies \eqref{slowlypower}. Let us perform degree-dependent percolation on this graph with edge-retention probability $p(d)$ as described in Definition \ref{halfper}. Then, if $p(d)$ satisfies
	\be\label{probcond}
	p(d)>b\exp\left\{-c(\log d)^{\gamma}\right\},
	\ee
with $\gamma<1$, and $b, c>0$, then the empirical degree distribution $F_n^{rr}$ of the degree sequence $\boldsymbol{d^{rr}}$ of the percolated graph also obeys a power law. More precisely, there exists $x_{0}$ such that, for all $x\in[x_{0},n^{\alpha})$, where $\alpha>\frac{1}{2}$, and w.h.p.,
	\be\label{condtight2}
 	1-F_{n}^{rr}(x)\geq \frac{c}{x^{\tau-1-C(\log x)^{\gamma-1}}},
	\ee 
with $\mathcal{L}_{n}=\sum_{i\in[n]}d_{i}\leq n\beta$ for some positive $\beta$ and a given $x_{0}>0$.\end{proposition}

Notice that \eqref{condtight2} is precisely the condition that is necessary to apply Proposition \ref{uppertight}. Thus, it allows us to give a uniform bound on the graph-distance in $\CMnd_{p(d)}$. Also, note that while the degrees $\boldsymbol{d^{r}}$ are still i.i.d.\ when $\boldsymbol{d}$ is, the degrees $\boldsymbol{d^{rr}}$ are {\em not} i.i.d. This explains why we needed to extend the conditions in Proposition \ref{uppertight}.

We prove Proposition \ref{csh} in two steps: First, in Lemma \ref{samepower}, we determine the distribution of the number of regular half-edges, i.e., the distribution of $\boldsymbol{d^{r}}=(d_1^r, \dots, d_n^r)$. Then, we investigate the final degree of regular vertices: we remove those regular half-edges that are paired to artificial vertices, thus determining the regular-to-regular degree sequence $\boldsymbol{d^{rr}}=(d_1^{rr}, \dots, d_n^{rr})$. Therefore, we can also refer to the graph $\CMnd_{p(d)}$ with the notation $\mathrm{CM}_n(\boldsymbol{d^{rr}})$. To prove Proposition \ref{csh}, we first prove the following lemma:

\begin{lemma}[Same power-law exponents after percolation]
\label{samepower}
Consider $\CMnd$ with i.i.d.\ degrees following distribution $F$ satisfying \eqref{slowlypower}. Suppose the conditions of Proposition \ref{csh} apply on $p(d)$, and let $(d_1^{r},\dots, d_n^{r})$ the i.i.d.\ degree sequence obtained by keeping a half-edge connected to a vertex with degree $d$ with probability $p(d)$, as defined in Definition \ref{halfper}.
Let $F^{r}(x)$ be the distribution function of $d_{i}^{r}$ for a single $i\in[n]$. Then, $F^{r}(x)$ still satisfies \eqref{slowlypower}.
Further, w.h.p., the empirical distribution function $F^{r}_n(x)$ satisfies \eqref{condtight}.
\end{lemma}

To prove Lemma \ref{samepower}, we will use the following lemma about concentration of binomial random variables:

\begin{lemma}[Concentration of binomial random variables]
\label{concenthop}
Let $R$ be a binomial random variable.
Then
	\be\Pv\left(R\notin [\mathbb{E}[R]/2, 2\mathbb{E}[R]]\right)\leq2\exp\{-\mathbb{E}[R]/8\}.
	\ee
\end{lemma}
\begin{proof}
See e.g., \cite[Theorem 2.19]{van2009random}.
\end{proof}

\begin{proof}[Proof of Lemma \ref{samepower}]
Let $\boldsymbol{d}$ be the degree sequence  in $\CMnd$, $p:=p(d)$ the edge-retention probability, and $\boldsymbol{d}^{r}, F^r, F_{n}^{r}$ as above. First note that, by construction, when 
$\boldsymbol{d}$ is i.i.d., then also $\boldsymbol{d}^{r}$ is. The upper bound is obvious since
\eqref{slowlypower} implies that
 	\be
 	1-F^r(x)=\Pv(d^r_i>x)\le \Pv(d_i>x) \le x^{-\tau+1+C(\log x)^{\gamma-1}}.
 	\ee
For the lower bound, for some $y=y(x)$ to be chosen later,
	\be\label{eq:lower1}
 	1-F^{r}(x)=  \Pv(\BIN(d_i,p(d_i))>x) \ge \Pv(\BIN(d_i,p(d_i))>x\mid d_i\geq y)\Pv(d_i\geq y).
	\ee
Suppose $y$ is such that $yp(y)\geq 2x$. Then the first factor on the rhs of the previous equation can be bounded as follows:
	\begin{equation*}
	\Pv(\BIN(d_i,p(d_i))>x|d_i\geq y)\geq\min_{z: z\geq y}\Pv(\BIN(z,p(z))>x)).
	\end{equation*}
Then, using the monotonicity of $p(d)$, it holds that  $zp(z)\geq yp(y)\geq 2x$, thus we can apply Lemma \ref{concenthop} to obtain
	\begin{equation*}
	\min_{z\colon z\geq y}\Big(\Pv(\BIN(z,p(z))>x\Big)\\\geq\min_{z\colon z\geq y}\Big(1-\exp\Big\{-\frac{zp(z)}{8}\Big\}\Big).
	\end{equation*}
Again, monotonicity of $p(d)$ implies  $zp(z)\geq yp(y)$, so
	\begin{equation*}
	\min_{z\colon z\geq y}\Big(1-\exp\Big\{-\frac{zp(z)}{8}\Big\}\Big)= 1-\exp\Big\{-\frac{yp(y)}{8}\Big\}.
	\end{equation*}
Combining this with \eqref{eq:lower1}, we obtain the lower bound
	\begin{equation}
	\label{maggiorazione}
	\Pv(\BIN(d_i,p(d_i))>x\mid d_i\geq y)\Pv(d_i\geq y)\geq \Big(1-\exp\Big\{-\frac{yp(y)}{8}\Big\}\Big)\frac{1}{y^{\tau-1+C(\log y)^{\gamma-1}}}.
	\end{equation}
Let us set now $y(x):=2x\exp\{2c(\log \frac{2}{b}x)^{\gamma}\}/b$, for $\gamma, c, b$ as in $\eqref{probcond}$. Since $p(y)$ satisfies the lower bound given in \eqref{probcond},
	\be\ba
     	y(x)p(y(x))&=2x \exp\left\{ 2c (\log \tfrac{2}{b}x)^{\gamma}-c \left( \log \tfrac{2}{b}x + 2c (\log \tfrac{2}{b}x)^{\delta}\right)^\gamma\right\}\\
     	&= 2x \exp\left\{ 2c (\log \tfrac{2}{b}x)^{\gamma} -c (\log \tfrac{2}{b}x)^{\gamma}\left(1+ 2c (\log \tfrac{2}{b}x)^{\gamma-1}\right)^\gamma \right\}.
	\ea
	\ee
Note that since $\gamma<1$, the factor $\left(1+ 2c (\log \tfrac{2}{b}x)^{\gamma-1}\right)^\gamma$ is  less than say $3/2$ (but larger than $1$) if $x$ is large enough, and hence,  for large enough $x$, the rhs is at least
	\[ 
	y(x)p(y(x)) \ge 2x \exp\left\{ \frac12 c (\log \tfrac{2}{b}x)^{\gamma}  \right\} \ge 2x.
	\]
This shows that we can indeed apply Lemma \ref{concenthop} above.

Note that, due to the bound $y p(y)\ge 2x$, the factor $(1- \exp\{ - y p(y)/8\})\ge 1/2$  for large enough $x$.
Using this estimate and again that $y=2 x\exp\{2c(\log \frac{2}{b}x)^{\gamma}\}/b$ we obtain from \eqref{maggiorazione} the lower bound on \eqref{eq:lower1}:
	\be
	1-F^{r}(x) \geq\frac12 \frac{1}{(\frac{2}{b}x)^{\tau-1}} \exp\left\{-(\tau-1)\big(2c(\log \tfrac{2}{b}x)^{\gamma}\big) - C \Big(\log(\tfrac{2}{b}x) + 2c (\log \tfrac{2}{b}x)^{\gamma}   \Big)^{\gamma} 
	\right\}.
	\ee
As before, the second term in the exponent is  $C (\log\tfrac{2}{b}x)^\gamma (1+ 2c (\log \tfrac{2}{b}x)^{\gamma-1})^\gamma$, and, since $\gamma<1$, the latter factor is at most $3/2$ when $x$ is sufficiently large.
Thus
	\be
	1-F^{r}(x) \geq\frac{b^{\tau-1}}{2^\tau} \frac{1}{x^{\tau-1}} \exp\left\{-(\tfrac32 C +(\tau-1))(\log \tfrac{2}{b}x)^{\gamma} \right\}.
	\ee
Finally, again for sufficiently large $x$, $\log \tfrac2b x= \log x (1+ \frac{\log(2/b)}{\log x})\le 2 \log x$, so we arrive at
	\eqn{
	\label{samepower-cdf}
	1-F^{r}(x) \geq\frac{b^{\tau-1}}{2^\tau} \frac{1}{x^{\tau-1}} \exp\left\{-2^\gamma(\tfrac32 C +(\tau-1))(\log x)^{\gamma} \right\}.
	}
As a result, we see that $F^{r}(x)$ satisfies the condition in \eqref{slowlypower} with exponent $\tau$, $\gamma\in (0,1)$, and  $C$ replaced by $2^\gamma(\tfrac32 C +(\tau-1)).$

It is not hard to show that when $d_i^{r}$ is i.i.d.\ with a distribution function $F^{r}(x)$ that satisfies \eqref{slowlypower}, then its empirical distribution function $F_n^{r}$
satisfies  \eqref{condtight2} for some $\alpha>1/2$. Indeed, by \eqref{samepower-cdf}, it follows that $d_{i}^{r}$ are i.i.d.\ random variables with distribution function $F^{r}$ satisfying 
\eqref{slowlypower}. Then $n(1-F^{r}_{n}(x))$ is a binomial random variable with parameters $n$ and $1-F^{r}(x)$, so that, by Lemma \ref{concenthop},
	\be\label{binomiale}
	\Pv\Big(1-F^{r}_{n}(x)\leq \frac{1}{2}\frac{n(1-F^{r}(x))}{n}\Big)<\exp\Big\{-n\Big(\frac{1-F^{r}(x)}{8}\Big)\Big\}.
	\ee
By the monotonicity of $1-F^{r}(x)$,
	\be\label{monbin}
	\exp\Big\{-n\Big(\frac{1-F^{r}(x)}{8}\Big)\Big\}\leq \exp\Big\{-n\Big(\frac{1-F^{r}(n^{\alpha})}{8}\Big)\Big\},
	\ee
for every $x\leq n^{\alpha}$. Then, since  we have just showed in \eqref{samepower-cdf} that $F^r$ satisfies \eqref{slowlypower}, for all $n$ sufficiently large, and $\al>1/2$ (but we are allowed to choose $\al<1/(\tau-1)$), 
	\be\label{lower-222}
	n(1-F^{r}(n^{\alpha}))>\frac{n}{n^{\al(\tau-1+\delta)}}\ge n^{c},
	\ee 
with $\delta$ arbitrarily small and some constant $c>0$.
Then, by a union bound,
	\be\label{eq:union-111}
	\Pv\Big(\exists x\leq n^{\alpha}\colon 1-F^{r}_{n}(x)\leq \frac{1}{2}\frac{n(1-F^{r}(x))}{n}\Big)\leq \sum_{x\leq n^{\alpha}}\Pv\Big(1-F_{n}^{r}(x)<\frac{1}{2}n[1-F^{r}(x)]\Big)
	\ee
and $\alpha$ is defined in \eqref{condtight}, and it is arbitrarily close to $\frac{1}{2}$. By \eqref{binomiale} and \eqref{monbin}, each summand is at most the rhs of \eqref{monbin}. Combining this with \eqref{lower-222} we obtain that each summand on the rhs of \eqref{eq:union-111} is at most $\exp{\{-n^{c}\}}$ for some $c>0$. So, we obtain
	\be 
	\Pv\Big(\exists x\leq n^{\alpha}\colon 1-F^{r}_{n}(x)\leq \frac{1}{2}(1-F^{r}(x))\Big)< n^{\alpha}\exp\{-n^{b}\}=o_n(1),
	\ee
as required. This finishes the proof of Lemma \ref{samepower}.
\end{proof}

We are now ready to complete the proof of Proposition \ref{csh}:
\begin{proof}[Proof of Proposition \ref{csh}]
Let $\boldsymbol{d^{r}}=(d_{1}^{r},\dots ,d_{n}^{r})$ be the degree sequence of the regular vertices of $\CMnd$ after the half-edge degree percolation and before having 
paired the half-edges. Let  $\boldsymbol{d^{rr}}=(d^{rr}_{1},\dots ,d^{rr}_{n})$ be the degree sequence of $\CMnd_{p(d)}$. Notice that (due to the pairing) the degrees 
$(d_{i}^{rr})_{i\in[n]}$ are no longer independent. We have that
	\be
	d^{rr}_{i}=d^{r}_{i}-\{\text{number of regular half-edges of vertex $i$ paired to artificial ones}\}.
	\ee
We define $\mathcal{R}_{n}$ as the total number of regular half-edges in the graph after the degree-dependent half-edge percolation.
We first prove that, for some $q_\ve>0$ and for all $\varepsilon>0$, w.h.p.,
	\be\label{upperbound}
	\frac{\mathcal{R}_{n}}{\mathcal{L}_{n}}\geq q_\varepsilon\ge 2q
	\ee
where $q_{\varepsilon}\rightarrow 3q$ as $\varepsilon\rightarrow 0$, for some $q>0$.
By definition,
	\be
\frac{\mathcal{R}_{n}}{\mathcal{L}_{n}}=\frac{\frac{1}{n}\sum_{i\in[n]}d_{i}^{r}}{\frac{1}{n}\sum_{i\in[n]}d_{i}}.
	\ee
By the Weak Law of Large Numbers, we have that for all $\varepsilon>0$ and w.h.p.
	\be\label{weakone}
		\mathbb{E}[D](1-\varepsilon)\leq \frac{1}{n}\sum_{i\in[n]}d_{i}\leq \mathbb{E}[D](1+\varepsilon),
	 \ee
and
	\be\label{weaktwo}
		\mathbb{E}[D^{r}](1-\varepsilon)\leq \frac{1}{n}\sum_{i\in[n]}d_{i}^{r}\leq \mathbb{E}[D^{r}](1+\varepsilon).
	 \ee
Then,
	\begin{align}\label{324}
	\Pv\Big(\frac{\mathcal{R}_{n}}{\mathcal{L}_{n}}\leq
	\frac{\mathbb{E}[D^{r}](1+\varepsilon)}{\mathbb{E}[D](1-\varepsilon)}\Big)\leq \Pv(\mathcal{R}_{n}\leq\mathbb{E}[D^{r}](1-\varepsilon))+\Pv(\mathcal{L}_{n}\geq \mathbb{E}[D](1+\varepsilon)).
	\end{align}
By \eqref{weakone} and \eqref{weaktwo}, the right hand side in \eqref{324} is $o_{n}(1).$
Then, with $q_{\varepsilon}:=\mathbb{E}[D^{r}](1-\varepsilon)/(\mathbb{E}[D](1+\varepsilon))$ and $3q:=\mathbb{E}[D^{r}]/\Ev[D]$, \eqref{upperbound} is established.
In what follows we prove that w.h.p.\ for all $x\in (x_0, n^{\alpha})$,
	\be\label{thirdstep}
	\Pv(\exists x\leq n^{\alpha}\colon 1-F_{n}^{rr}(x)\leq \frac{1}{2}[1-F_{n}^{r}(b x)])=o_{n}(1).	
	\ee
By Lemma \ref{samepower} we know that $F_n^{r}$ satisfies \eqref{condtight}, thus establishing \eqref{thirdstep} is sufficient  to show that $F_n^{rr}$ also satisfies \eqref{condtight}.

Let $\mathcal{S}^{rr}_{n}(x)$,  $\mathcal{S}^{r}_{n}(x)$  be the sets of vertices with degree higher than $x$ in $\mathrm{CM}_n(\boldsymbol{d^{rr}})$ and $\mathrm{CM}_n(\boldsymbol{d^{r}})$ respectively, and let $\mathcal{E}^{r}_{n}(x)$, $\mathcal{E}^{rr}_{n}(x)$ be the set of half-edges in the set $\mathcal{S}^{r}_{n}(x)$ and $\mathcal{S}^{rr}_{n}(x)$. If $q$ is defined in \eqref{upperbound}, we show that for all $\varepsilon>0$, there exists $x_0$ such that the distribution of the degrees of vertices in $\mathcal{S}_{n}^{rr}(x_0)$ can be bounded from below with an i.i.d.\ sequence of Binomial random variables having distribution $\BIN(d^{r},q)$.

For this, we consider the pairing procedure of the half-edges in $\mathcal{E}^{r}_{n}(x)$ as described in Definition \ref{halfper}, and give a \emph{uniform lower bound} on the probability that a half-edge in $ \mathcal{E}^{r}_{n}(x)$ is paired to a regular half-edge. By the interchangeability of the pairing of the half-edges, we have the following bound:
For any $1\leq k\leq |\mathcal{E}_{n}^{r}(x)|$ we consider an ordered sequence of half-edges $s_{1},\dots ,s_{k}\in\mathcal{E}_{n}^{r}(x)$, 
	\be\label{pk}
	\mathcal{P}_{k}=\Pv(s_k\text{ is paired to a regular half-edge}\mid \mathcal{F}_{k-1})\geq \frac{\mathcal{R}_{n}-2k}{\mathcal{L}_{n}-2k}
	\geq \frac{\mathcal{R}_{n}-2|\mathcal{E}^{r}_{n}(x)|}{\mathcal{L}_{n}},
	\ee
where $\mathcal{F}_{k-1}$ is the $\sigma-$algebra for the first $k-1$ pairings.
By the definition of $\mathcal{E}^{r}_{n}(x)$, for all $\varepsilon>0$, there exists $x_0$ s.t. for all $x\ge x_0$, w.h.p.\ 
 	\be\label{fracEn}
 	\frac{|\mathcal{E}^{r}_{n}(x_0)|}{\mathcal{L}_{n}}\leq\frac{\varepsilon}{2}.
 	\ee
From the proof of \eqref{upperbound}, $\mathcal{R}_{n}/\mathcal{L}_{n}\ge 2q$, implying that w.h.p. 
	\be\label{pk-lower}
	\mathcal{P}_{k}\ge 2q-\varepsilon \ge q,
	\ee
where $3q=\Ev[D^r]/\Ev[D]>0$ as before. This gives a uniform lower bound on the pairing probability independently of the outcome of the previous pairings.
As a consequence, for any $v\in \mathcal{S}_{n}^{r}(x_0),$  the degree $d_{v}^{rr}$ can be coupled to a random variable $\mathcal{X}_{v}$ such that
	\be \label{eq:xiv}
\mathcal{X}_{v}\buildrel{d}\over{\geq}\BIN(d_{v}^{r},q),
	\ee
and for any choice of vertices $w$ and $w'$ in $\mathcal{S}_{n}^{r}(\bar{x})$, $\mathcal{X}_{w}$ and $\mathcal{X}_{w'}$ are independent.
Therefore, w.h.p.,
	\be 
	\Pv(d_{v}^{rr}>m)\geq \Pv\Big(\BIN(d^{r}_{v},q)>m\Big).
	\ee
We now prove \eqref{thirdstep}.
For this, let us consider 
	\be\label{frr}
	1-F_{n}^{rr}(x)\geq \frac{1}{n}\sum_{i\in[n]}\mathbbm{1}_{\{d_{i}^{r}>\frac{10x}{q}\}}\mathbbm{1}_{\{d_{i}^{rr}>x\}}
	\ee
where $x\geq x_0$.
Using \eqref{eq:xiv}, we get 
	\be\label{indices}
	\frac{1}{n}\sum_{i\in[n]}\mathbbm{1}_{\{d_{i}^{r}>\frac{10x}{q}\}}\mathbbm{1}_{\{d_{i}^{rr}>x\}}\buildrel{d}\over{\geq}
	\frac{1}{n}\sum_{i\in[n]}\mathbbm{1}_{\{d_{i}^{r}>\frac{10x}{q}\}}\mathbbm{1}_{\{\BIN_i(\frac{10x}{q},q)>x\}},
	\ee 
	where $\BIN_i(\frac{10x}{q},q)$ is a shorthand notation for independent binomials with the given parameters.
We can rewrite the right hand side of \eqref{indices} as
	\be\label{frri}
	[1-F_{n}^{r}]\left(\frac{10x}{q}\right)-\frac{1}{n}\sum_{i\in[n]}\mathbbm{1}_{\{d_{i}^{r}\geq\frac{10x}{q}\}}
	\mathbbm{1}_{\{\BIN_i(\frac{10x}{q},q)<x\}}.
	\ee
Denoting $n(1-F_{n}^{r}(\frac{10x}{q}))=:m$, and define $X:=\sum_{i=1}^{m}\mathbbm{1}_{\{\BIN_i(\frac{10x}{q},q)<x\}}$. We aim to show that $X<m/2$ holds whp.
By Lemma \ref{concenthop},
	\be 
	\Pv\Big(\BIN\big(\frac{10x}{q},q\big)\leq x\Big)\leq \e^{-10x/8}< \e^{-5x_0/4},
	\ee
so that 
	\be	
	\mathbb{E}[X]=\mathbb{E}\left[\sum_{i=1}^{m}\mathbbm{1}_{\{\BIN_i(\frac{10x}{q},q)<x\}}\right]\leq m\e^{-5x_0/4}.
	\ee
	Let us increase $x_0$ if necessary to ensure that $\e^{-5x_0/4}\le 1/8$.
Note that $X$ is a sum of Bernoulli random variables.  Thus, we can use a Chernoff bound to bound the upper tail of $X$ (see, e.g., \cite[Corollary 3.13]{Sinc11}, with $\beta=\tfrac{1}{2}\exp\{5x_0/4\}-1\ge 3$ and $\mu=m\exp\{-5x_0/4\}$, to obtain the following bound for some $b',C'>0$:
	\be\ba\label{chernoff111}
	\Pv\Big(X>\frac{m}{2}\Big)&<\e^{-\beta^2\mu/(2+\beta)}\leq \exp\left\{-\frac{(\e^{\frac{5x_0}{4}}/2-1)m\e^{-\frac{5x_0}{4}}}{2}\right\}\\
	&<\e^{-m/8}\leq \exp\{-n^{b'}C'\},
	\ea\ee
where first we used that $\beta>2$ thus $\beta^2/(2+\beta)> \beta/2$, and then again, $1/2-\e^{-5x_0/4}\ge 1/4$ to obtain the first formula in the second line. Finally, since $m\ge n^{b'}$ as long as $x\le n^{1/(\tau-1-\delta)}$, we obtain the desired bound above. Then, noting that $m=n [1-F_{n}^{r}](10x/q)$, by \eqref{frr} and \eqref{frri}, \eqref{chernoff111} yields
	\be 
	\Pv\Big(1-F_{n}^{rr}(x)\leq \frac12[1-F_{n}^{r}](10x/q)\Big)\leq \exp{\{-n^{b'}C'\}},
	\ee
so that, by a union bound, 
	\be 
	\Pv\Big(\exists x \in (x_0,n^{\alpha}): 1-F_{n}^{rr}(x)\leq \frac{1}{2}[1-F_{n}^{r}](10x/q)\Big)\leq n^{\alpha}\exp{\{-n^{b'}C'\}}.
	\ee
Then, for $x_0$ large enough, \eqref{thirdstep} follows with $b=10/q$.
This concludes the proof of Proposition \ref{csh}.
\end{proof}

We now use Proposition \ref{csh} to obtain a bound on the graph-distance in the graph after percolation.
Let $\mathcal{D}_{n}^{rr}$ be the graph-distance in $\mathrm{CM}_n(\boldsymbol{d^{rr}})$.
Then, we have the following bound:

\begin{corollary}\label{percolateddistance}
Given $\CMnD$ with i.i.d.\ degrees having distribution satisfying \eqref{slowlypower}, and the corresponding graph $\widetilde{\mathrm{CM}}_n(\boldsymbol{d})_{p(d)}$ that is the result of degree-dependent half-edge percolation with $p(d)$ satisfying \eqref{probcond}. Let $\boldsymbol{d}^{rr}$ denote the corresponding degree sequence for the vertices. Then, for all $\varepsilon_{\ref{percolateddistance}}>0$, there exists $k=k(\varepsilon_{\ref{percolateddistance}})$ s.t., if $u$ and $v$ are two uniformly chosen vertices in $[n]$ with degree at least $k$,
	\be\label{uppertightdistance}
	\sup_{n\geq 1}\Pv\Big(\mathcal{D}_{n}^{rr}(u,v)-\frac{2\log \log n}{|\log (\tau-2)|}\ge1\mid d_{u}\geq k, d_{v}\geq k \Big)< \varepsilon_{\ref{percolateddistance}}.
	\ee
\end{corollary}
\begin{proof}
Let  $A_{n}$ be the event for $\al>1/2$ as in Proposition \ref{csh}
	\be\label{suffcond}
	A_{n}:=\Big\{1-F^{rr}_{n}(x)\geq \frac{\tilde{C}}{x^{\tau-1+C_{1}(\log x)^{\gamma-1}}},\ \forall x_0\leq x\leq n^{\alpha}\Big\},
	\ee
where $x_0$ is defined in \eqref{fracEn}.
Then by Proposition \ref{csh}, $\Pv(A_{n})=1-o_{n}(1)$. For brevity let us write $\CK_n:=\{\mathcal{D}_{n}^{rr}(u,v)-\frac{2\log \log n}{|\log (\tau-2)|}\ge1\}$. Then,
	\begin{align}
	\sup_{n\geq 1} & \ \Pv\Big(\CK_n\mid d_{u}\geq k, d_{v}\geq k \Big)=\\
	\sup_{n\geq 1}&\left\{\Pv\Big(\CK_n\mid d_{u}\geq k, d_{v}\geq k ,A_{n}\Big)\Pv(A_{n})+
	\Pv\Big(\CK_n\mid d_{u}\geq k, d_{v}\geq k ,A_{n}^{c}\Big)\Pv(A_{n}^{c})\right\}.
	\end{align}
Note that the probability of the first term on the rhs is bounded precisely in Proposition \ref{uppertight}, on $A_n$, while $\Pv(A_n^c)=o_n(1)$ by Proposition \ref{csh}. Thus, there exists $\varepsilon_{\ref{uppertight}}$ s.t.  
	\begin{equation}
	\sup_{n\geq 1} \Pv\Big(\CK_n\mid d_{u}\geq k, d_{v}\geq k \Big)\le \varepsilon_{\ref{uppertight}}(1-o_{n}(1))+o_{n}(1)\le \varepsilon_{\ref{percolateddistance}},
	\end{equation}
	for all $n$ sufficiently large, finishing the proof.
\end{proof}

\section{Tightness of the weight}
\label{sec-tightness}Our goal is to prove Theorem \ref{integralcondition} in this section.
We first prove a uniform bound for the weight of the path that connects the uniformly chosen vertex $u$ to a vertex in the first layer $\Gamma_{y_{0}}$ in $\CMnd_{p(d)}$. For this, we prove that for any choice of $p(d)$ that satisfies condition \eqref{probcond} we can connect a uniform vertex $u$ in $\CMnd$ to a vertex that has sufficiently large degree after the degree-dependent half-edge percolation. We introduce the following notation:
	\be
	\partial \mathcal{B}_{m}(u):=\{w\in \CMnd \text{ s.t. }\mathcal{D}_{n}(u,w)= m\},
	\ee
and
	\begin{align} \label{degs}
	&\deg(V_{m}(u)):=\max_{w\in \partial \mathcal{B}_{m}(u)}d_w,\\
	&\deg(V_{m}^{\sss {p(d)}}(u)):=\max_{w\in \partial \mathcal{B}_{m}(u)}d_w^{rr},
	\end{align}
	where recall that  $d_w^{rr}$ equals the degree of $w_{p(d)}$ in $\CMnd_{p(d)}$.
The following holds:

\begin{lemma}\label{tightshortdis}
For any choice of $p(d)$ that satisfies \eqref{probcond}, for all $\varepsilon_{\ref{tightshortdis}}>0$ and $k>0$, there exists a constant $m:=m(k,\varepsilon_{\ref{tightshortdis}})$ such that w.h.p. 
	\be
	\Pv(\deg(V_{m}^{\sss{p(d)}}(u))< k)<\varepsilon_{\ref{tightshortdis}}.
	\ee	
	\end{lemma}
\begin{proof}
The proof consists of two steps.
In the first step we show that for a fixed half-edge retention probability $p(d)$, when a vertex has sufficiently high degree before percolation, then after the degree-dependent percolation its degree is at least $K$ with high probability.
In the second step we prove that we can find such a vertex a bounded number of steps away from the uniformly chosen vertex $u$.

 As we have seen in Lemma \ref{samepower} and the proof of Proposition \ref{csh}, the degree $d^{rr}_w$ of a vertex $w$ after percolation is the result of a half-edge percolation as described in Lemma \ref{samepower} and of a second thinning with parameter at least $q$ as described in \eqref{pk} and \eqref{pk-lower}, whenever $d_w\ge x_0$. This results in a two-step thinning of $d_w$, where $d_w^r\ {\buildrel d \over =}\  \BIN(d_w, p(d_w))$ and $d_w^{rr}\ {\buildrel d \over \ge }\ \BIN(d_w^r, q)$. Thus we obtain the stochastic domination $d_w^{rr}\ {\buildrel d \over \ge }\  \BIN(d_w, p(d_w)q)$, for all $w$ with $d_w\ge x_0$.

Note that since $p(d)$ satisfies \eqref{probcond}, and $q$ is fixed, the mean $dp(d)q$ is monotone increasing in $d$ for all $d$ sufficiently large.  
Indeed, 
\be \mathbb{E}[\BIN(d,\e^{-c(\log d)^{\gamma}})q]\ge db\e^{-c(\log d)^{\gamma}}q = \e^{\log d(1-c(\log d)^{\gamma-1})}q.\ee
Thus, let us define for any $z>2$, $\widetilde K_z:=\inf\{d: db\e^{-c(\log d)^{\gamma}}q \ge zk\}$.
Then,  a Chernoff bound as in \cite[Corollary 13.3]{Sinc11}, with $\beta=1-1/z\ge 1/2$ and $\mu=zk$ yields that
\[ \Pv( \BIN(\widetilde K_z, p(\widetilde K_z) q) < k ) \le \e^{-\beta^2 \mu/2} \le \exp\{ -(1-1/z)^2 zk/2\} \le \exp\{-zk/8\}.  \]
Thus, for any $\varepsilon_{0}>0$ and any $k>0$, we can choose $z:=z(\ve_0)$ large enough and then $\widetilde K_{z(\ve_0)}$ accordingly, such that the rhs is at most $\ve_0$. Combining this with the stochastic domination and the monotonicity argument above, we obtain that for any vertex $w$
	\be\label{bound} 
	\Pv( d_w^{rr} < k | d_w\ge  \widetilde K_{z(\ve_0)} )<\varepsilon_{0}.
	\ee 
We now show that we can connect a uniformly chosen vertex $u$ with a vertex with degree at least $\widetilde{K}_{z(\ve_0)}$ in a bounded number of steps and probability tending to $1$.
By \cite[Proposition 2]{2015arXiv150601255B},
	\be\label{ak}
	|\partial \mathcal{B}_{m}|\convp\infty,
	\ee
as $m\rightarrow\infty$. Further, the sequence $\{(d_{v_i})_{v_{i}\in\partial \mathcal{B}_{m}}\}$ can be coupled to $|\partial \mathcal{B}_{m}|$ many i.i.d.\ random variables with a power-law distribution described in \eqref{forwarddegree} (see, e.g., \cite[Proposition 4.7]{BHH10}).
Then 
	\be 
 	\Pv(\deg(V_{m}(u))\leq \widetilde{K}_{z(\ve_0)})\leq \Big(\frac{\alpha_{1}}{\widetilde{K}_{z(\ve_0)}^{\tau-2+\alpha_{2}}} \Big)^{m},
 	\ee 
 with $\deg(V_{m}(u))$ defined in \eqref{degs} and for some positive constants $\alpha_{1}$ and $\alpha_{2}$.
 Then, for all $\varepsilon_{1}>0$, there exists $m_{\varepsilon_{1}}$ such that $\Pv(\deg V_{m_{\varepsilon_{1}}}\leq \widetilde{K}_{z(\ve_0)})<\varepsilon_{1}$.
Finally, let 
	\begin{align*} 
	&E_{1}:=\{d_w^{rr} < k | d_w\ge  \widetilde K_{z(\ve_0)}\}\\
	&E_{2}:=\{\deg(V_{m_{\varepsilon_1}}(u))\leq\widetilde{K}_{z(\ve_0)}\},
	\end{align*}
Then, by a union bound,
	\be 
	\Pv(\deg(V_{m}^{p(d)})< k)\leq \Pv(E_{1}^{c})+\Pv(E_{2}^{c})\leq \varepsilon_{0}+\varepsilon_{1}.
	\ee
with $\varepsilon_{0}$ and $\varepsilon_{1}$ arbitrarily small.
This completes the proof of the lemma, with distance $m=m_{\ve_1}$.
\end{proof}

We are ready to prove Theorem \ref{integralcondition} by constructing a path with tight excess weight.

\begin{proof}[Proof of Theorem \ref{integralcondition}]
Let us set $k>x_{0}$, where $x_{0}$ is defined in Proposition \ref{uppertight}.
As a consequence of Lemma \ref{tightshortdis}, for any fixed $\varepsilon$, with probability at least $1-\varepsilon$, there exists a path in $\CMnd$ that connects $u$ with a vertex $u_{k}$ with  $d_u^{rr}\ge k$  in at most $m:=m(\varepsilon_{\ref{tightshortdis}}, k)$ steps. We call this path $\pi_{k}$.
The weight over the path $\pi_{k}$ is given by the sum of at most $m$ i.i.d.\ random variables with distribution $F_{Y}$. Therefore for all $\varepsilon_{\eqref{pesok}},$ there exists $r'=r'(\varepsilon_{\eqref{pesok}})$ such that
	\be\label{pesok}
	\Pv(W_{n}(u,u_{k})>r')<\varepsilon_{\eqref{pesok}}
	\ee
We recursively describe a path $\pi_{u}(i)$ in $\CMnd_{p(d)}$ as follows: $\pi_{u}(0)=u_{k}$, $\pi_{u}(i+1)$ is the vertex having the maximum degree between all the neighbours of $\pi_{u}(i)$. If there are more vertices with the same degree we choose the one that is connected to $\pi_{u}(i)$ with the least edge-weight.
Then, if $y_{i}$ with $y_{0}=k$, is a sequence as in \eqref{recursive}, the proof of Proposition \ref{uppertight} guarantees that $\pi_{u}(i)$ has degree at least $y_{i}$ for all $i$, with probability $1-o_k(1)$, see \eqref{varepsiloni}. We define $\pi_v(i)$ analogously.
We aim to prove an upper bound on the weight of the path that connects $u_k$ and $v_k$   in $\CMnd_{p(d)}$. For this recall that we have assigned a weight from distribution $1+X$ to \emph{each half-edge} and that the half-edge percolation described in Definition \ref{halfper} can be later interpreted as thinning the half-edges with too high excess weight. Nevertheless,
	\be
	W_{n}(u_{k},v_{k})\leq \mathcal{D}^{rr}_{n}(u_{k},v_{k})+\sum_{i=0}^{b(n)}\Big(w_{\pi_{u}(i)}^{1}+w_{\pi_{u}(i)}^{2}+ w_{\pi_{v}(i)}^{1}+w_{\pi_{v}(i)}^{2}\Big),
	\ee
where for $z=u,v$, $w_{\pi_{z}(i)}^{1}, w_{\pi_{z}(i)}^{2}$ are the excess edge-weight of the two half-edges attached to $\pi_z(i)$ that the path $\pi_z$ uses.
We now fix $\varepsilon>0$, and we choose $k$ so large that $\sum_{i=1}^{b(n)}\varepsilon_{i}<\varepsilon$. 
Considering two identical constructions for both vertices $u_{k}$ and $v_{k}$, by Proposition \ref{uppertight}, with probability at least $1-\varepsilon_{\ref{uppertight}}$, $\mathcal{D}^{rr}_{n}(u_{k},v_{k}) \le 2\frac{\log\log n}{|\log(\tau-2)|}$.

Further, again by the \emph{proof of} Proposition \ref{uppertight}, that $\pi_{z}(i)$ is a vertex with degree at least $y_i$ \emph{in the percolated graph}.
Recall that we can write $p(d)=\Pv(X\le \xi_d)$ for some deterministic $(\xi_d)_{d\ge 1}$ and the half-edge percolation can be interpreted as thinning the half-edges with too high excess weight. Thus, the additional edge weights on the half-edges attached to $\pi_{z}(i)$ survived the degree-dependent percolation and as a result are at most $\xi_{y_{i}}$.
This implies that $w_{\pi_{z}(i)}^{j}\leq \xi_{y_{i}}$ for $j=1,2$ and $z=u,v$.
Therefore,
	\be\label{peso}
W_{n}(u_{k},v_{k})\leq 2\frac{\log\log n}{|\log(\tau-2)|}+4\sum_{i=0}^{b(n)}\xi_{y_{i}}.
	\ee
Finally, by \eqref{pesok}, the probability that the sum of the weights in the paths connecting $u$ with $u_{k}$ and $v$ with $v_{k}$ does not exceed $2r'$ is at least $1-2\varepsilon_{\eqref{pesok}}$.
Then,
	\be\label{weight-upper-final} 
	W_{n}(u,v)\leq 2r'+2\frac{\log\log n}{|\log(\tau-2)|}+2\sum_{i=0}^{b(n)}\xi_{y_{i}},
	\ee 
with probability at least $1-2\varepsilon_{\eqref{pesok}}-\varepsilon_{\ref{uppertight}}$.

It remains to show that there exists a choice of $\xi_d$ such that $\sum_{i=0}^{b(n)}\xi_{y_{i}}$ is bounded.

First we start rewriting the integrability criterion in \eqref{inverse-crit}. By a change of variables $u:=1/y$, we obtain that \eqref{inverse-crit} is equivalent to the convergence of
\be\label{inverse-changed}\int_0^c F^{(-1)}_{\sss X}\left( \exp\{ -C/y \}\right) \frac1y \mathrm dy \ee
for some $c,C>0$. Note first that, due to a simple change of variables, the constant $C$ in the exponent can be chosen arbitrarily. We shall set its proper value later on.
For now, let us fix an \emph{arbitrary} $\alpha \in (\tau-2, 1)$ and cut the integral at the powers of $\alpha$. Then, the convergence of \eqref{inverse-changed} implies the convergence of the sum
	\be
	\label{eq::bound11} 
	\sum_{n=K}^\infty \int_{\alpha^{n+1}}^{\alpha^{n}} F^{(-1)}_{\sss X}\left( \exp\{ -C/y \}\right) \frac1y \mathrm dy,
	\ee
where $L$ can be chosen as $K:=\min \{n\colon \al^n \le c\}$.
By the monotonicity of the inverse function $F_{\sss X}^{(-1)}(\cdot)$,
	\be
	\label{eq::bound12}  
	F^{(-1)}_{\sss X}\left( \frac{1}{\e^{C/\alpha^{n+1}}}\right) (1-\alpha) \le \int_{\alpha^{n+1}}^{\alpha^{n}} F^{(-1)}_{\sss X}\left( \frac{1}{\e^{C/y}}\right) \frac1y \mathrm dy,
	\ee
hence the convergence of the integral in \eqref{inverse-crit} implies that
	\be
	\label{eq::sum-3} 
	\sum_{n=K}^{\infty} F^{(-1)}_{\sss X}\left( \frac{1}{\e^{C/\alpha^n}}\right) < \infty.
	\ee
Now we turn to the choice of $\xi_d$ and thus $p(d)$.
The intuitive idea is the following: recall that in the path constructed in the proof of Proposition \ref{uppertight},  by Lemmas \ref{yibounds} and \ref{stim}, for the $i$th vertex of the constructed path, the degrees are $y_i\in ((k^{1-\delta})^{1/(\tau-2)^i}, k^{1/(\tau-2)^i})$, for some constants $k$ and  $\delta\in(0,1)$. Thus, we would like to set $p(d)$ so that the equation
	\be
	\label{eq::initial} 
	p\left(y_i\right)= \e^{-C/\alpha^i}
	\ee
holds.
For this, $d=(k^{1-\delta})^{1/(\tau-2)^i}$ implies that $i=\log \left( \log (d)/\log (k^{1-\delta})\right)/{|\log(\tau-2)|}$, which, when used on the right hand side of \eqref{eq::initial}, results in the definition
	\[ 
	p(d):=\exp \left\{ -C\left( \frac{\log d}{\log (k^{1-\delta})}\right)^{ |\log \alpha|/|\log (\tau-2)|} \right\} =  \exp \left\{ -C'\left( \log d\right)^{ |\log \alpha|/|\log (\tau-2)|} \right\},
	\]
with $C':=C \log (k^{1-\delta})^{-|\log \alpha|/|\log (\tau-2)|}$. Due to the fact that we have chosen $\alpha\in(\tau-2,1)$, we have  $|\log \alpha|/|\log (\tau-2)|:=\gamma<1$.
Thus, the conditions of Proposition \ref{csh} are satisfied with this choice of $p(d)$.
Further, note that, since $p(d)=\Pv(X\le \xi_d)$, we have $\xi_d=F_{\sss X}^{(-1)}(p(d))$.
Since $p(\cdot)$ is monotone decreasing in $d$, $y_i \ge (k^{1-\delta})^{1/(\tau-2)^i}$ implies
\[ \xi_{y_i}=F_{\sss X}^{-1}(p(y_i)) \le F_{\sss X}^{(-1)}\left(p((k^{1-\delta})^{1/(\tau-2)^i})\right)= F_{\sss X}^{(-1)}\left(\e^{-C/\alpha^i}\right). \]
Note that the right hand side is summable in $i$ by \eqref{eq::sum-3}, thus the excess edge-weight on the edges in the path through the layers $\Gamma_i$ is bounded, i.e.
\be
\sum_{i=0}^{\infty} \xi_{y_i}< \infty.
\ee
This, combined with \eqref{weight-upper-final} finishes the proof of the upper bound. 
The proof of the lower bound follows from the fact that 
\[ W_n(u,v) \ge \CD_n(u,v)\]
and the latter is tight around $2\log\log n/|\log (\tau-2)|$  by \cite{MR2318408}.
 \end{proof}

\paragraph{Acknowledgements}
This work is supported by the Netherlands
Organisation for Scientific Research (NWO) through VICI grant 639.033.806 (EB and RvdH), VENI grant 639.031.447 (JK), the Gravitation {\sc Networks} grant 024.002.003 (RvdH) and the STAR Cluster (JK).

\bibliographystyle{abbrv}
\bibliography{bibliografia}

\end{document}